\documentclass[12pt,twoside]{article}

\usepackage{setspace}
\usepackage{amssymb}
\usepackage{amssymb}
 \usepackage{amsthm}
\usepackage{amsmath}
\usepackage{graphicx}
\usepackage{color}
\usepackage{amsmath}
\usepackage{amssymb}
\usepackage{amsthm}
\usepackage{graphicx,color}
\usepackage{pstricks}
\usepackage{pst-tree}
\usepackage{multido}
\usepackage{comment}
\def\R{{\mathbb R}}
\def\dd{{\mathrm d}}
\def\N{{\mathbb N}}
\def\PP{{\mathbb P}}
\def\cU{{\mathcal U}}
\def\cB{{\mathcal B}}
\def\cH{{\mathcal H}}
\def\cF{{\mathcal F}}

\def\dd{{\rm d}}

\newtheorem{thm}{Theora}[section]
\newtheorem{theo}[thm]{Theorem}
\newtheorem{hypothesis}[thm]{Hypothesis}

\newtheorem{prop}[thm]{Proposition}
\newtheorem{Def}[thm]{Definition}
\newtheorem{rem}[thm]{Remark}
\newtheorem{example}[thm]{Example}

\begin{document}

 \title{Explicit invariant measures for infinite dimensional SDE driven by L\'evy noise with dissipative nonlinear drift I}
  \def\lhead{S.\ ALBEVERIO, B. Smii, L.\ Di Persio} 
  \def\rhead{Invariant measures for stochastic differential equations driven by L\'evy noise} 


\author{Sergio Albeverio\footnote{Dept. Appl. Mathematics, University of Bonn, HCM, BiBoS, IZKS. {\tt albeverio@uni-bonn.de}} %
  \\
  Luca Di Persio\footnote{University of Verona, Department of Computer Science, strada Le Grazie, 15, Verona,
    Italia. {\tt luca.dipersio@univr.it}}
    \\
  Elisa Mastrogiacomo\footnote{Universit\'a degli Studi di Milano Bicocca,
   Dipartimento di Statistica e Metodi Quantitativi,
   Piazza Ateneo Nuovo, 1 20126 Milano
    Italia. {\tt elisa.mastrogiacomo@unimib.it}}
   \\
 Boubaker Smii\footnote{King Fahd University of
Petroleum and Minerals, Dept. Math. and Stat., Dhahran 31261, Saudi
Arabia. \hspace*{5mm}{\tt boubaker@kfupm.edu.sa}}}
\date{}
\maketitle

\begin{abstract}
We stu\dd y a class of nonlinear stochastic partial differential
equations with dissipative nonlinear drift, driven by L\'evy noise.
Our work is divided in two parts. In the present part I we first
define a Hilbert-Banach setting in which we can prove existence and
uniqueness of solutions under general assumptions on the drift and
the L\'evy noise. We then prove a decomposition of the solution
process in a stationary component and a component which vanishes
asymptotically for large times in the $L^p-$sense, $p\geq1$. The law
of the stationary component is identified with the unique invariant
probability measure of the process.

In part II we will exhibit the invariant measure as the limit of
explicit invariant measures for finite dimensional approximants.

We shall present a general discussion of such explicit invariant
measures involving, in particular, ground state transformation for
L\'evy driven processes and relations to invariant measures for
Ornstein-Uhlenbeck like processes with L\'evy noise.

Examples and applications are also provided.

\end{abstract}
\vskip0.2cm \noindent {\bf Key words:} {SPDEs, dissipative drift,
invariant measures, dissipative systems, process driven by L\'evy
noise, stationary components, Gibbs-type measures, explicit
invariant measures, ground state transformation.}
\section{Introduction}
Stochastic differential equations for processes with values in
infinite dimensional spaces have been studied in the literature
under different assumptions on the coefficients and the driving
noise term. They are intimately related with stochastic partial
differential equations, looking upon the processes as taking values
in
 the infinite dimensional state space expressing their dependence on the space variable. Among the by now numerous books on these topics for Gaussian noise let us
  mention \cite{DPZVerde,DPZRosso}, \cite{GaMa}, \cite{PRRO}, \cite{Holden-et-all}, see also, e.g., \cite{BrzPe}, \cite{Marin,MarQuer}, \cite{RozMi}. For the case of non Gaussian noises see \cite{PeZa},
  \cite{MandrekarRudiger}. \\
In \cite{ALEBOU} a stu\dd y was initiated concerning a class of non
linear stochastic differential equations with L\'{e}vy noise and a
drift term consisting of a linear unbounded space-dependent part
(typically a Laplacian) and an unbounded non linear part of the
dissipative type and of at most polynomial growth at infinity.
 \\
This class is of particular interest since it contains the case of
Fitz Hugh Nagumo equations with space dependence, on a bounded
domain of $\mathbb{R}^n$ or on bounded networks with 1-dimensional
edges. Such equations are of interest in a number of areas including
neurobiology and physiology, see, e.g., \cite{AlDiP},
\cite{Tu1,Tu2}, \cite{Tu}. It is also related with the stochastic
quantization equation in quantum field theory, mostly studied with
Gaussian noise, see \cite{JoMi}, \cite{S.Mitter},
\cite{AlRo},\cite{DaPraDeb}, \cite{DaPTu}, \cite{ALKR},
\cite{AKaMi}, and references therein. See also \cite{ALGYo} for
related equations with L\'{e}vy noise. For the SPDE equations of the
FitzHug-Nagumo type with L\'evy noise
 studied in \cite{ALEBOU}, existence and uniqueness of solutions was proven, as well as existence and uniqueness of invariant measures. Moreover asymptotic small
 noise expansion for the solutions have been established in the same paper \cite{ALEBOU}. \\
In the present paper we further stu\dd y the framework set up in
\cite{ALEBOU} and provide a decomposition of the solution process in
the sum of a stationary component and a component which vanishes
asymptotically for large times in the $L^p$-sense, $1\le p<\infty$,
with respect to the underlying
 probability measure. \\
The law of the stationary component is then identified with the unique invariant measure of the solution process. \\
This is a an extension of a result that was proved before in the case of Gaussian noise in \cite{ADPM}. \\
In part II we shall show that the invariant measure is the limit of
explicitly given finite dimensional  measures and we relate them in
certain cases to the
invariant measure for an Ornstein-Uhlenbeck process driven by a
finite dimensional L\'{e}vy process in $\mathbb{R}^n$.

We shall provide a general discussion of explicit invariant measures
for L\'evy driven processes in finite dimensions extending in
particular the linear drift case of \cite{Sat91, SY84} to the case
of non lilnear drift.

For the derivation of these relations we shall use an adaptation to
the case of L\'{e}vy noise of methods developed before in a
diffusion setting (with Gaussian noise) which go under the
name of $h$-transform, see \cite{Dy}, or ``ground state transformation'', see, e.g.,  \cite{AHKS} and  \cite{ReSi}. \\
In the finite dimensional case this has been discussed  in
\cite{A-Cufaro}, where the general case
of finite dimensional L\'{e}vy noise has been considered. \\
We relate in particular the work in \cite{A-Cufaro} to work on
solutions of martingale problems and relations to weak solutions of
L\'evy driven SDE.
Our explicit invariant measures in finite dimensions considerable extend previously known examples in \cite{ABRW}, \cite{BaNie}. \\
The structure of part I of this series of two papers is as follows. \\
In Section 2 the setting for the infinite dimensional stochastic differential equations with linear drift and  L\'{e}vy noise is described. Cylindrical L\'evy processes are hereby introduced and the invariant measure is discussed.\\
In Section 3 the results on existence and uniqueness of solutions of such equations with non linear drift and of corresponding invariant probability measure  are recalled. \\
In Section 4 the basic theorem giving the additive decomposition of
the solution process in a stationary and an asymptotically small
component is formulated and proven. Its proof uses a double indexed
finite dimensional approximant, one, $m$, referring to a Yosida
approximation of the non linear term (which we allow to
be non globally Lipschitz, typically with polynomial growth at infinity!) and the other one, $n$, referring to $n$-dimensional approximations of the SPDE. \\
In the proof results on stochastic convolution (referring to \cite{PeZa}) and Gronwall's type estimates are exploited. \\

Applications to the study of models described by PDE's of the
FitzHugh-Nagumo type with L\'{e}vy noise will be presented in part
II.

\section{The infinite dimensional Ornstein-Uhlenbeck process driven by L\'evy noise}\label{sec:OU}
There is an increasing interest in the study of stochastic evolution
equations driven by L\'evy noise. In this section we concentrate on
the linear stochastic differential equation
\begin{equation}\label{eq:OU}
\begin{aligned}
   &\dd X(t)= A X(t) \dd t + \dd L(t), \qquad t\geq 0,\\
   & X(0)=x \in \cH,
\end{aligned}
\end{equation}
where $\cH$ is a real separable Hilbert space, $(L(t))_{t\geq 0}$ is
an infinite dimensional cylindrical symmetric L\'evy process and $A$
is a self-adjoint operator generating a  $C_0$-semigroup. In
particular we are going to recall a few results concerning the
well-posedness of the above equation (Sec. \ref{sec:cylLevy}) and
the existence and uniqueness of explicit invariant measures (Sec.
\ref{ssec:invOU}), in the case where $A$ satisfies Hypothesis
\ref{hp:A} below, mainly following \cite{PrZa}.

\subsection{Cylindrical L\'evy process}\label{sec:cylLevy}
Let us recall that a L\'evy process $(L(t))_{t\geq 0}$ with values
in a real separable Hilbert space $\cH$ is a process on some
stochastic basis $(\Omega,\cF,(\cF_t)_{t\geq 0},\mathbb{P})$, having
stationary independent increments, c\`adl\`ag trajectories in $\cH$,
and such that $L(0)=0$, $\PP$-a.s. One has that
\begin{align}\label{eq:Lseries}
    \mathbb{E}[e^{i\langle L(t),u\rangle}]= \exp (-t\psi(u)), \qquad u\in \cH,\:, t\geq 0 \:,
\end{align}
where the exponent $\psi$ can be expressed by the following infinite
dimensional L\'evy-Khintchine formula:
\begin{align}\label{eq:Kintchineformula}
   \psi(u)=\frac{1}{2}\langle Qu,u\rangle - i \langle a,u \rangle - \int_{\cH} \left( e^{i\langle u,y\rangle}-1 - \frac{i\langle u,y\rangle}{1+|y|^2}\right) \nu(\dd y), \qquad u\in \cH.
\end{align}
Here $Q$ is  a symmetric non-negative trace class operator on $\cH$,
$a\in \cH$ and $\nu$ is the L\'evy measure or  jump intensity
measure associated to $(L(t))_{t\geq 0}$, i.e. $\nu$ is a
$\sigma$-finite Borel measure on $\cH$ such that $\nu(\left\{
0\right\})=0$ and
\begin{align*}
    \int_{\cH} (|y|^2 \wedge 1) \nu(\dd y) <\infty.
\end{align*}
We shall say that $L(t)$ is generated by the triplet $(Q, \nu, a)$.
Let us remark that the third term on the right hand side of
(\ref{eq:Kintchineformula}) can also be written with the term $-
\frac{i\langle u,y\rangle}{1+|y|^2}$ replaced by $- i \langle u,y
\rangle\,\chi_{D}(y)$ with $D$ the unit ball in $\cH$ and $a$
replaced by $y$.

 In this paper we will consider a cylindrical L\'evy process
$L=(L(t))_{t\geq 0}$ defined by the orthogonal expansion
\begin{align}\label{eq:L(t)-operator}
    L(t)=\sum_{n=1}^\infty \beta_n L^n(t) e_n,
\end{align}
where $e_n$ is an orthonormal basis in $\cH$, $L^n=(L^n(t))_{t\geq
0}$, $L^n(0)=0$, are independent real valued, symmetric, identically
distributed    L\'evy processes without a Gaussian part, defined on
a fixed stochastic basis. Hence we are assuming, in particular,
that $L(t)$  has $Q\equiv 0$ and $a\equiv 0$. Moreover, $\beta_n$ is
a given (possibly unbounded) sequence of positive real numbers.

We also assume that $(e_n)_{n\in \N}$ is made of the eigenvectors of
the leading operator $A$ in \eqref{eq:OU}, assumed to have purely
discrete spectrum. More precisely we make the following assumptions:
\begin{hypothesis}\label{hp:A}
 $A$ is a self-adjoint, strictly negative operator with domain $D(A)$, which generates a $C_0$-semigroup $e^{tA}$, $t \geq0$, on $\cH$
 and such that there is a fixed basis $(e_n)_{n\in \N}$ in $\cH$ verifying:
     $(e_n)_{n\in \N} \subset D(A)$, $Ae_n=-\lambda_n e_n$, with $\lambda_n>0$,
     for any $n \in \mathbb{N}^+$ and $\lambda_n \uparrow +\infty$.
\end{hypothesis}

\begin{rem}\label{eq:remark}
   From the above assumptions it follows that:
    \begin{enumerate}
       \item the action of $e^{tA}$ on any element $u\in \cH$ can be written as
       \begin{align*}
           e^{tA} u= \sum_{k=1}^{+\infty} e^{-\lambda_k t}\langle u,e_k\rangle, \qquad k\in \N, \ t\geq 0.
       \end{align*}
       \item Since the law of $L^n, n\in \N$ is assumed to be symmetric, independent of $n$ we have, for any $n \in \mathbb{N}^+, t\geq 0$,
       \begin{align*}
            \mathbb{E}[e^{ihL^n(t)}] = e^{-t \psi_\R(h)}, \quad h\in \R,
       \end{align*}
       with $\psi_\R$ being given by
       \begin{align}\label{eq:nu}
          \psi_\R(h)= \int_{\R} (1-\cos(hy) \nu_\R(\dd y), \quad h\in \R,
       \end{align}
       and the L\'evy measure $\nu_\R$ associated with $L^n$ is symmetric for any $n \in \mathbb{N}$ (i.e. $\nu_\R(A) =\nu_\R(-A)$, for any Borel subset $A$ of $\R$).
    \end{enumerate}
\end{rem}


The definition of cylindrical Levy process in
\eqref{eq:L(t)-operator} is only formal, since it has of course to
be supplied with suitable assumptions on the $\beta_i$. We need
namely to give conditions under which the series on the right hand
side of \eqref{eq:L(t)-operator} converges in $\cH$.
Anyway, we can prove that $(L(t))_{t\geq 0}$ is always a
well-defined Levy process with values into a suitable Hilbert space
$\cU$. To this end, we recall that any infinite dimensional
separable Hilbert space $\cH$ can be identified with the space
$\ell^2$, using the basis $(e_n)_{n\in \N}$
In general, for a given sequence $\rho=(\rho_n)_{n\in \N}$ of real
numbers, we set
\begin{align*}
    \ell^2_\rho:= \left\{ (x_n) \in \R^\N: \ \sum_{n\geq 1} x^2_n \rho^2_n <\infty\right\}.
\end{align*}
The space $\ell^2_\rho$ becomes a Hilbert space with the inner
product $\langle x,y\rangle:= \sum_{n\geq 1} x_ny_n\rho^2_n$ for
$x=(x_n),y=(y_n) \in \ell^2_\rho$. For $\rho_n=1$, $\forall n \in
\mathbb{N}$, we have $\ell^2_\rho = \ell^2$. We quote from
\cite[Proposition 2.4]{PrZa} the result which provides conditions on
$\beta_n, \nu_\mathbb{R}$ such that the cylindrical L\'evy process
of the form given in \eqref{eq:L(t)-operator} is well-defined in
some Hilbert space $\cU$.
\begin{prop}\label{prop:cyl-Levy}
    The following conditions are equivalent:
    \begin{align*}
        &(i) \qquad \sum_{n=1}^\infty (\beta_n L^n(t_0))^2 <+\infty \qquad for \ some \ t_0 >0 \:, a.s. \:;\\
        & (ii) \qquad \sum_{n=1}^\infty (\beta_n L^n(t))^2 <+\infty \qquad for \ any \ t >0 \:, a.s. \:; \\
        &(iii) \qquad \sum_{n=1}^\infty \left( \beta^2_n \int_{|y|<1/\beta_n} y^2 \nu_\mathbb{R}(\dd y) + \int_{|y|\geq 1/\beta_n} \nu_\mathbb{R}(\dd y)\right)<
                +\infty.
    \end{align*}
    \end{prop}

\begin{rem}\label{rem:lrho}
    According to Proposition \ref{prop:cyl-Levy}, our cylindrical L\'evy process $L$ is a L\'evy process taking values in the Hilbert space $\cU:=\ell^2_\rho$, with a properly chosen weight
$\rho$. More precisely, we can choose any sequence $\rho=(\rho_n)$
such that
\begin{align*}
    \qquad \sum_{n=1}^\infty \left( \rho^2_n\beta^2_n \int_{\rho_n\beta_n|y|<1} y^2 \nu_\mathbb{R}(\dd y) + \int_{\rho_n\beta_n|y|\geq 1} \nu_\mathbb{R}(\dd y)\right)<
                +\infty.
\end{align*}
\end{rem}
Now let us come back to the Ornstein-Uhlenbeck process described by
equation \eqref{eq:OU}. According to Hypothesis \ref{hp:A}, we may
consider equation \eqref{eq:OU} as an  infinite sequence of
independent one dimensional equations, i.e.
\begin{align}\label{eq:finitedimensioneq}
 \begin{cases}
    \dd X^n(t)= -\lambda_n X^n(t) \dd t + \beta_n L^n(t),\\
    X^n(0)=x_n, \ n\in \N,
    \end{cases}
\end{align}
with $x=\sum_{n} x_n e_n \in \cH$, and $\left( x_n \right)_{n \in
\mathbb{N}} \in l^2\left( \mathbb{R} \right)$. The solution of
\eqref{eq:finitedimensioneq} is the stochastic process
$X=\sum_{n\geq 1} X^n e_n$, with components
\begin{align*}
   X^n(t)=e^{-\lambda_n t}x_n + \int_0^t e^{-\lambda_n (t-s)} \beta_n \dd L^n(s), \qquad n \in \N, \ t\geq 0.
\end{align*}
The processes $X^n$, for $n\in \mathbb{N}$, can be assumed to be
almost surely right-continuous with left limits. Using point (2) in
Remark \eqref{eq:remark} and following \cite[pag. 105]{Sat91}, we
can compute their characteristic functions. In particular we have
\begin{align*}
   \mathbb{E}[e^{ihX^n(t)}] = \exp \left( ie^{-\lambda_n x_n h}+ \int_0^t \psi_\R (e^{-\lambda_n(t-s)}\beta_nh)\dd s\right), \qquad h\in \R, t\geq 0, n \in \mathbb{N},
\end{align*}
where $\psi_\R$ is the function defined in \eqref{eq:nu}. In another
way, we write
\begin{align}
   &X(t) = e^{tA}x + L_A(t), \quad \label{eq:OUmild}
   \intertext{where}
   &L_A(t)= \int_0^t e^{(t-s)A} \dd L(s) = \sum_{n=1}^\infty \left( \int_0^t
   e^{-\lambda_n(t-s)} \beta_n \dd L^n(s)\right) e_n \notag,
\end{align}
and the process $X$ is $\cF_t$-adapted and Markovian. Here the
convergence of the series in the last member of the equality above
is to be meant in probability. More details on the solution of
equation \eqref{eq:OU} are given in \cite[Theorem 2.8]{PrZa}. In the
following we provide an application of this result to the
construction of an invariant probability measure for
Ornstein-Uhlenbeck processes driven by a quite general class of
symmetric cylindrical L\'evy noises, which we shall call, for
simplicity, O-U-L\'evy processes.


\subsection{Invariant measure for the infinite dimensional O-U-L\'evy driven processes}\label{ssec:invOU}
As usual we say that a probability measure $\mu$ on a complete
separable metric (i.e.\ polish) space $\cH$ is invariant with
respect to a Markov semigroup $(P_t)_{t\geq 0}$, with transition
probability kernel $P_t(x,\mathrm{d}y)$, $x,y\in \cH$ on $\cH$ if
for any Borel subset $\Gamma\subset E$ and any $t\geq0$ we have
$\mu(\Gamma)=\int\mu(\mathrm{d}x)P_t(x,\Gamma)$. We say shortly
$\mu$ is an invariant measure for $(P_t)_{t\geq 0}$. See, e.g.,
\cite[chapter 16]{PeZa} for equivalent formulations of this
property.
\begin{prop}
   Assume Hypothesis \ref{hp:A}. Moreover, assume that $\beta_n$, $ n \in \N$, is a bounded sequence and that the symmetric L\'evy measure $\nu_\R$ appearing in \eqref{eq:nu} satisfies
   \begin{align*}
        \int_1^{+\infty} \log(y)\nu_\R(\dd y) <\infty.
   \end{align*}
   Finally, assume that
   \begin{align*}
       \sum_{n=1}^\infty \frac{1}{\lambda_n} <\infty.
   \end{align*}
   Then the L\'evy driven Ornstein-Uhlenbeck process $X=(X(t))_{t\geq 0}$ given by \eqref{eq:OUmild}
   admits a unique invariant measure, in the sense that the process $X(t)$ is invariant under the Markov transition semigroup associated with $X(t)$.
\end{prop}

\begin{proof}
    The proof proceeds basically on the same line of \cite[Proposition 2.11]{PrZa}.
   To show that there exists    an invariant measure we first notice that (according to \cite[Theorem 17.5]{Sato}) that each one dimensional Ornstein-Uhlenbeck process $X^n(t)$ has an invariant measure $\mu_n$ which is the law of the random variable
   \begin{align*}
       \int_0^\infty e^{-\lambda_n u } \beta_n \dd L^n(u),
   \end{align*}
   having characteristic function $$
   \hat{\mu}_n(h)= \exp\left( -\int_0^\infty \psi_\R(e^{-\lambda_n u} \beta_n h)\dd u\right),
   \qquad h\in \R. $$
   Let us consider the product measure $\mu = \Pi_{n\geq 1} \mu_n$ on $R^\N$. This is the law of the family $(\xi_n)_{n\in \N}$ of independent  random variables, where
   \begin{align*}
      \xi_n = \int_0^\infty e^{-\lambda_n u} \beta_n \dd L^n(u), \quad n\geq 1.
   \end{align*}
    We underline that $\xi_n$ is an infinite divisible real-valued random variable.
   Now define $\xi=\sum_{n=1}^\infty \xi_n e_n$. Then the random variable $\xi$ takes values in $\cH$ if and only if
   the L\'evy measures $\nu_n$ of $\xi_n$ verify
   \begin{align}\label{eq:integrabilitycondition}
       \sum_{n=1}^\infty \int_\R (1 \wedge y^2) \nu_n(\dd y) <\infty.
   \end{align}
    The latter comes from the definition of $\xi_n$ and the condition (iii) in Proposition \ref{prop:cyl-Levy}.
   Exploiting the result in \cite[Proposition 2.11]{PrZa}, we deduce that   condition \eqref{eq:integrabilitycondition} is satisfied, hence we have $\mu(\cH)=1$.
   It is possible to prove that $\mu$ is the unique invariant measure of $X$ given by \eqref{eq:OUmild}, by showing that, for any $x\in \cH$,
   \begin{equation}\label{eq:claiminvariant}
       \lim_{t\to +\infty} X(t)= \xi
   \end{equation}
   in probability.
   To prove this we first assume $x=0$ in \eqref{eq:OUmild}. In this case
   \begin{align*}
       X^n(t)= \int_0^t e^{-\lambda_n(t-s)} \beta_n \dd L^n(t).
   \end{align*}
   Now we recall  the following identity:
   \begin{align*}
       \mathbb{E}\left[e^{ih\int_s^t g(u)\dd L^n(u)} \right]=
       \exp \left( \int_s^t \psi_\R(g(u)h)\dd u\right) \:,\: h \in \mathbb{R} \:,\: 0\leq s \leq t \:,
   \end{align*}
   which holds for any real continuous function $g$ on $[s,t]$, see \cite[pag. 105]{Sato}.
   Then if we compute the characteristic function of $\xi_n$ we see that
   $X^n(t)$ and $\xi_n(t)$ have the same law.
   This fact allows us to estimate the following quantity:
   \begin{align*}
       a_\epsilon(t):= \mathbb{P}(|X(t)-\xi|^2 >\epsilon) \:,\: \text{ for any }\, \epsilon >0 \; , \; t >0.
   \end{align*}
   In fact, since $X^n$ and $\xi_n$ have the same law we can write
   \begin{align*}
       |X(t)-\xi|^2 &= \left| \int_0^t e^{(t-s)A}\dd L(s)- \sum_{n=1}^\infty\xi_n(t) e_n \right|^2 \\
       & = \left| \sum_{n=1}^\infty \int_0^t e^{-\lambda_n(t-s)}\beta_n\dd L^n(s)- \sum_{n=1}^\infty \int_0^\infty e^{-\lambda_n s}\beta_n\dd L^n(s) \right|^2 \\
       &= \left| \sum_{n=1}^\infty\int_0^t e^{-\lambda_n s}\beta_n\dd L^n(s)- \sum_{n=1}^\infty \int_0^\infty e^{-\lambda_n s}\beta_n\dd L^n(s) \right|^2\\
       &=\left| \sum_{n=1}^\infty\beta_n^2\int_t^\infty e^{-\lambda_n s}\beta_n\dd L^n(s) \right|^2.
   \end{align*}
   Moreover, (see \cite[pag. 11]{PrZa}) it is possible to prove that the law of the random variable in the last term in the previous expression coincides with the one of
   \begin{align*}
      \sum_{n=1}^\infty e^{-2\lambda_n t}\xi_n^2.
   \end{align*}
   We then have, for any $t>0$,
   \begin{align*}
        a_\epsilon(t):= \mathbb{P}\left(  \sum_{n=1}^\infty e^{-2\lambda_n t}\xi_n^2 >\epsilon\right) \leq \mathbb{P} \left( e^{-\lambda_1 t} \sum_{n=1}^\infty \xi_n^2 >\epsilon\right) = \mathbb{P}(|\xi|^2 > e^{-2\lambda_1 t} \epsilon).
   \end{align*}
   Letting $t\to +\infty$ we find $\lim_{t\to +\infty } a(t)=0$, for any $\epsilon > 0$. This proves the claim \eqref{eq:claiminvariant} for the case $x=0$. The general case, $x \neq 0$, can easily be obtained by translation.
\end{proof}

\begin{rem}\label{rem:LAinfty}
   For further use (see Section \ref{sec:section2}), we emphasize that from the proof of the claim \eqref{eq:claiminvariant}  we can also see that,
   for any $t>0$,
   the random variables
   \begin{align*}
       \int_0^t e^{(t-s)A} \dd L(s) \qquad and \qquad \xi_t:=\sum_{n=1}^\infty \int_0^t e^{-\lambda_n s} \beta_n \dd L^n(s)
   \end{align*}
   have the same law.
   Hence we obtain that the random variable
   \begin{align*}
      L_A(+\infty):= \lim_{t\to +\infty } \int_0^t e^{(t-s)A} \dd L(s)
   \end{align*}
   is well-defined and its law coincides with the one of the random variable
   $\xi$.
\end{rem}

\section{The stochastic semilinear differential equation}
\label{sec:section1}
\subsection{The state equation: existence and uniqueness of solutions}\label{ssec:section1}
In this section we concentrate on a semilinear SDE driven by L\'evy
noise, thus extending the setting of Section \ref{sec:OU} to include
a non linear drift. Following basically the setting of
\cite{ALEBOU} and \cite{PeZa}, we consider a stochastic differential
equation of the form
\begin{equation}
    \label{e1}
\left\{ \begin{array}{ll}
            \mathrm{d}X(t) = AX(t)\,\mathrm{d}t + F(X(t))\,\mathrm{d}t + B\,\mathrm{d}L( t), \quad t \geq 0 &\\
            X(0) = x \in D(F)
\end{array}
\right.
\end{equation}
where the stochastic process $X=(X(t))_{t\geq0}$ takes values in a
real separable Hilbert space $\mathcal{H}$, $A$ is a linear operator
from a dense domain
 $D(A)$ in $\mathcal{H}$ into $\mathcal{H}$
which generates a $C_0$-semigroup of strict negative type.

$B$ is a linear bounded operator from a suitable Hilbert space $\cU$
(which is to be precised, see Remark \ref{rem:cyl-levy} below) into
$\mathcal{H}$. $F$ is a mapping from $D(F)\subset\mathcal{H}$ into
$\mathcal{H}$, continuous, nonlinear, Fr\'{e}chet differentiable and
such that
\begin{equation}
    \langle F(u)-F(v)-\eta(u-v),u-v \rangle < 0, \quad for \ some \ \eta >0
\end{equation}
and all $u,v \in D(F)$, where $\langle \, , \, \rangle$ is the scalar product in $\mathcal{H}$. \\
The connection between $A$ and $F$ consists in requiring that
$\omega>\eta$ (so that $A+F$ is maximal dissipative or
$m$-dissipative in the sense of \cite[pag. 73]{DPZVerde}, i.e. the
range of $\lambda - (A+F)$ is $\cH$, for some - and consequently all
- $\lambda >0$).

We assume that $L$ is a cylindrical L\'{e}vy process according to
the description given in Section \ref{sec:cylLevy}, i.e.
\begin{align*}
     L(t)= \sum_{n=1}^\infty \beta_n L^n(t)e_n,
\end{align*}
with $(\beta_n)_{n\in \N}$ a given (possibly unbounded) sequence of
positive real numbers, $(L^n)_{n\in \N}$ a sequence of real valued,
symmetric, i.i.d. L\'evy processes without Gaussian part and
$(e_n)_{n\in \N} $ an orthonormal basis of $\cH$. Moreover, we will
work under Hypothesis \ref{hp:A}, hence assuming that $(e_n)_{n\in
\N} $ is made of eigenvectors of $A$.

\begin{rem}\label{rem:cyl-levy}
   According to Remark \ref{rem:lrho}, identifying $\cH$ with the space $\ell^2$, the cylindrical L\'evy process $L(t)$ is a well-defined L\'evy process taking values in the Hilbert space $\cU:=\ell^2_{\rho}$ for a suitable sequence $\rho = \left( \rho_n \right)_{n \in \mathbb{N}}$. This means that
\begin{align*}
   \sum_{k=1}^\infty (\beta_k^2 L^n(t)^2\rho_n^2) <\infty.
\end{align*}
  We assume that $B$ is a linear bounded operator acting on $\cU$ with values in $H$.
Through the identification of $\cH$ with $\ell^2$ we can think that
$B$ can be written in the following form:
\begin{align}
    B L(t)= \sum_{k=1}^\infty b_k \beta_k L^{n}(t)e_k,
\end{align}
with $(b_k)_{k\in \N}$ such that $\sup_{n\in \N} \frac{b_k}{\rho_k}
<\infty$.
\end{rem}

$L$ can be assumed to have c\`{a}dl\`{a}g paths (as follows e.g. by
\cite{APD} and \cite[p.39]{PeZa}). We assume as in
\eqref{eq:L(t)-operator} that it is a pure jump process, in the
sense that no Gaussian and deterministic part is present. In what
follows $\cB$ is a reflexive Banach space continuously embedded into
$\cH$ as a dense Borel subset. We assume that $A_{|\cB}$, $F_{|\cB}$
are almost $m$-dissipative in $\cB$, in the sense that
$A_{|\cB}+\omega \mathbf{1}_{|\cB}$ is $m$-dissipative in $\cB$.
Moreover we assume that $D(F) \supset \cB$ and that $F$ maps bounded
subsets of $\cB$ into bounded subsets of $\cH$. We also assume that
the stochastic convolution $L_A(t)$ of $B\,\mathrm{d}L(t)$ with
$S(t)=e^{tA}$, i.e.\
\begin{equation}\label{eq:LAcadlag}
    L_A(t) = \int_{0}^{t}S(t-s)B\,{d}L(s),\, t \geq 0,
\end{equation}
has a c\`{a}dl\`{a}g version in $D(-A_{\cB})^{\alpha}$ for some
$\alpha \in [0,1)$. Finally, we impose the following condition: for
all $T>0$
\begin{equation} \label{eq:intF}
    \int_0^T |F(L_A(t))|_{\cB} \dd t <\infty, \quad \PP-a.s.
\end{equation}
(cf. \cite[p. 183]{PeZa} for more details).
It follows from our assumptions that $L_A$ is square-integrable,
$\mathcal{F}_t$-adapted (where $\mathcal{F}_t$ is the natural
 $\sigma$-algebra associated with $L$) (see \cite[p. \ 163]{PeZa}).

\begin{rem}\label{Remark1}
    We give two cases where the property \eqref{eq:LAcadlag} of $(L_A(t))_{t\geq 0}$ holds:
    \begin{enumerate}
        \item  $\mathcal{B}$ is a Hilbert space and $S(t)$ is a contraction semigroup in this space $\mathcal{B}$ and $B\,L(t)$ takes values in
         $\mathcal{B}$; see, e.g.\ \cite[p.158, Theorem
         9.20]{PeZa}.
        \item $S(t)$ is analytic in $\mathcal{B}$ and $BL$ has c\`{a}dl\`{a}g  trajectories
        in  $D\left((-A_{\mathcal{B}})^{\alpha}\right)$, for some $\alpha\in [0,1)$, where $A_{\mathcal{B}}$ is the restriction of
         $A$ to $\mathcal{B}$, see, e.g.\ (\cite[p.\
         163, Prop. 9.28]{PeZa}).
    \end{enumerate}
\end{rem}


\begin{example}
    Let us provide an example for the setting $(\mathcal{H},\cU,\mathcal{B},L,A,Q,F)$where both conditions \eqref{eq:LAcadlag} and \eqref{eq:intF} hold. 
    Let $\Lambda\subset\mathbb{R}^n$, bounded and open, $n\in\mathbb{N}$, and let $\mathcal{H}=\cU:=L^2(\Lambda)$. Let $F$ be of the form of a multinomial of odd
    degree $2n+1$, $n\in\mathbb{N}$, i.e.\ $F$ is a mapping of the form $F(u)=g_{2n+1}(u)$, where $g_{2n+1}: \R \to \R$, is a polynomial of degree $2n+1$ with first derivative bounded from above,
    see \cite{ALEBOU}.
It follows that  $D(F)=L^{2(2n+1)}(\Lambda)\subsetneqq L^2(\Lambda)$. 
    We take $\mathcal{B}:=L^{2p}(\Lambda)$ with $p\geq 2n+1$. Let $A=\Lambda$ be the Laplacian in $L^2(\Lambda)$ with Neumann boundary conditions  on the boundary
    $\partial\Lambda$. Let $Q= \mathbf{1}$ as an operator in $\cH$ and let $L$ be a  process on $\mathcal{H}= L^2(\Lambda)$ of the type described in Remark \ref{Remark1},  and such that the corresponding L\'evy measure $\nu$ satisfies
    \begin{equation}\label{eq:norma}
      \int_{L^2(\Lambda)} \left| x \right|_{W^{\beta,2p(2n+1)}}\nu(\mathrm{d}x)<+\infty,
      \end{equation}
       where $ W^{\beta,2p(2n+1)}$ is a fractional Sobolev space
    with given index $\beta>0$; moreover we require $\int_{ \left| y \right|\le1} \left| y \right|^2\nu(\mathrm{d}y)+\nu(y\vert \left| y \right|\geq1)<\infty$.
    From \eqref{eq:norma} (see \cite[Prop. 6.9]{PeZa}) it follows that
    $L(t)\in D\left((-A_{2p(2n+1)})^{\gamma}\right)\subset L^{2p(2n+1)}(\Lambda)$ for some $\gamma>0$ and has c\`{a}dl\`{a}g trajectories in $ D\left((-A_{2p(2n+1)}^{\gamma})\right)$. Here $A_{2p(2n+1)}$ denotes
    the generator of the heat semigroup with Neumann boundary conditions operating in $L^{2p(2n+1)}$. Then by \cite[Prop. 9.28, p.163 and Theorem 10.15, pag.187]{PeZa}, $L_A$ is well-defined and satisfies \eqref{eq:LAcadlag} and  \eqref{eq:intF}.
\end{example}

Now we are rea\dd y to state the main result of this section, which
concerns with existence and uniqueness of solutions for equation
\eqref{e1}. We refer to \cite[Theorem 4.9]{ALEBOU} for the proof.

\begin{theo}
    \label{t1}
    Assume $A,Q,F,L$ satisfy all previous assumptions and let $L_A$ satisfy conditions \eqref{eq:LAcadlag} and \eqref{eq:intF} above. Then there exists a unique c\`{a}dl\`{a}g mild solution of
    \eqref{e1} in the sense of being adapted, c\`{a}dl\`{a}g in $\mathcal{B}$, for any $x\in\mathcal{B}$ and satisfying almost surely
    \begin{equation}
        X(t) = S(t)\,x + \int\limits_0^tS(t-s)F(X(s))\,\mathrm{d}s + L_A(t), \, t \geq 0, \ x \in D(F),
    \end{equation}
    with $X(t)\in D(F)$ for all $t\geq0$.

    \noindent For each $x\in\mathcal{H}$ there exists a generalized solution of \eqref{e1}, i.e. there exists $(X_n)_{n\in\mathbb{N}}$,
    $X_n\in\mathcal{B}$, $X_n$ unique mild adapted solutions of \eqref{e1} with $X_n(0)=x$ such that $\left| X_n(t)-X(t) \right|_{\mathcal{H}}\to0$ on each
    bounded interval, as $n\to\infty$.

    \noindent Moreover $X(t)$ defines Feller families on $\mathcal{B}$ and on $\mathcal{H}$, in the sense that the Markov semigroup $P_t$ associated with $X(t)$ maps for any
     $t\geq0$, $C_b(\mathcal{H})$ into $C_b(\mathcal{H})$ and $C_b(\mathcal{B})$ into $C_b(\mathcal{B})$.
\end{theo}
\subsection{Existence and uniqueness of an invariant measure}
\label{sec:section2} In the following we deal with the asymptotic
behavior of the Markov semigroup corresponding to equation
\eqref{e1}. We will see that in our case, in addition to existence
and uniqueness of the invariant measure $\mu$ (whose notion has been
recalled in Section \ref{sec:OU}), we can prove that it is also
exponentially mixing. We quote from \cite[chapter 16]{PeZa} the
definition of exponentially mixing invariant measure.
Let ${Lip}(\cH)$ be the space of all real valued Lipschitz
continuous functions $\psi:\cH\to\mathbb{R}$ endowed with norm $\|
\psi \|_{\infty} + \|  \psi \|_{Lip}$, with $\| \psi \|_{\infty}$
the sup-norm and
\begin{equation}
    \| \psi \|_{Lip} := {\sup_{x\not=y}} \frac{\left| \psi(x)-\psi(y) \right|}{\left| x-y \right|_\cH};
\end{equation}
$\| \psi \|_{Lip}$ is said to be the smallest Lipschitz constant for
$\psi$, see \cite[p.16]{PeZa}. We say that an invariant measure
$\mu$ is exponentially mixing with exponent $\omega>0$ and bound
function $c:\cH\to(0,+\infty)$ with respect to a Markov semigroup
$(P_t)_{t\geq 0}$ if
\begin{equation}
    \left| P_t\psi(x) - \int_\cH\psi(y)\,\mathrm{d}\mu (y) \right| \le c(x)\mathrm{e}^{-\omega t}\| \psi \|_{Lip}, \quad  \forall x \in \cH,\,
     \forall t > 0, \, \psi\in{Lip}(E).
\end{equation}
If $\mu$ is exponentially mixing, then $P_t(x,\Gamma)\to\mu(\Gamma)$
as $t\uparrow +\infty$ for any Borel subset $\Gamma$ of $E$ (cfr.\
\cite{DPZVerde} and \cite{PeZa}, p.\ 288).

We have the following from \cite[Theorem 16.6 p.293]{PeZa}:

\begin{theo}\label{t2}
    Let us consider the SDE \eqref{e1} under the assumptions on $A$ and $F$ given at the beginning of the Subsection \ref{ssec:section1} and under the assumptions of Theorem \ref{t1}.
    Assume, in addition:
   $${\sup_{t\geq0}}\,\mathbb{E} \left( \left| L_A(t) \right|_{\mathcal{H}} + \left| F\left(L_A(t)\right) \right|_{\mathcal{H}} \right) <+\infty.$$
    Then there exists a unique invariant measure $\mu$ for the Markov semigroup $(P_t)_{t\geq 0}$ on $\mathcal{H}$ associated with the mild solution $X$ of \eqref{e1}
    ($(P_t)_{t\geq 0}$ gives the transition probabilities for $X$). \\
    $\mu$ is exponentially mixing with exponent $\omega+\eta$ and a bound function $c$ of linear growth in the sense that
    \begin{equation}
        \left| c(x) \right| \le C(\left| x \right|+1),
    \end{equation}
    for some constant $C>0$ and for all $x \in \mathcal{H}$.
\end{theo}

\section{Decomposition of the solution process in a stationary and an asymptotically small component}

 Let us consider the cylindrical L\'evy process $L=(L(t))_{t\geq 0}$, as in Sections \ref{sec:OU} and \ref{sec:section1}. Following
\cite[p. 295]{PeZa}, let $\overline{L}(t)$, $t\in\mathbb{R}$ be the
corresponding double-sided process such that $\overline{L}(t)=L(t)$,
$t\geq0$ and $\overline{L}(-t)$, for $t\geq0$ is a process
independent of $L(t)$, $t\geq0$ and such that all finite dimensional
distributions of $\overline{L}(-t)$ coincide with those of $L(t)$,
$t\geq0$.

Our aim now is to split the solution $X$ of the equation \eqref{e1}
into the sum of a stationary process $r$ and an asymptotically (for
$t\rightarrow +\infty$) vanishing process $v$. To this end we will
split the  solution $X^{(n)}_m,\,n, m\,\in\,\N, $ of the
approximating equation into the sum of a stationary process
$r^{(n)}_m$ and an asymptotically  vanishing process $v^{(n)}_m$
satisfying some suitable properties.\\
Let $\{ {\cal B}_{n}\}_{n \in \mathbb{N}}$ be a sequence of finite
dimensional subspaces of the Banach space ${\cal B}$ introduced in
Sect.3  and $\{ \Pi_{n}\} $ a sequence of self-adjoint operators
from ${\cal H}$ onto ${\cal B}_{n}$ such that $\Pi_{n}x\rightarrow
x$ in ${\cal B}$, for arbitrary $x\in {\cal B}$. The existence of
such spaces can be proved as in \cite{Za88}, Prop.3.

Moreover let $F_{m}, m \in \mathbb{N}$ be the $m$-th Yosida
approximation of $F$ (i.e. $F_{m}:=m(mI-F)^{-1}$). We know that
$F_{m}$ is Lipschitz continuous and it satisfies the following
estimates:
\[|F_{m}(x)-F(x)|_{\cH}\rightarrow 0, \quad \quad m\rightarrow \infty, x \in D(F)\]
\[|F_{m}(x)|_{\cH}\leq |F(x)|_{\cH}, \quad \quad x \in D(F), \forall m \in \mathbb{N}\]

and
\[\langle F_{m}(x)-F_{m}(y),x-y\rangle \leq m|x-y|^{2}, \quad \forall m \in \mathbb{N}.\]

For any $n,m\in \mathbb{N}$ we consider the following families of
equations
\begin{align}\label{eq:EDSnm}
   \begin{cases}
     \dd X^{(n)}_m(t) = AX^{(n)}_m(t)\dd t+\Pi_nF_m(\Pi_n X^{(n)}_m(t)) \dd t +B\dd L(t), \\
     X^{(n)}_m(0)=\Pi_n x \in \cH,
   \end{cases}
\end{align}
and
\begin{align}\label{eq:EDSn}
   \begin{cases}
     \dd X^{(n)}_m(t) = AX^{(n)}_m(t)\dd t+\Pi_nF(\Pi_n X^{(n)}(t)) \dd t +B\dd L(t), \\
     X^{(n)}(0)=\Pi_n x\in \cH,
   \end{cases}
\end{align}
which can be seen as  approximating problems relative to \eqref{e1}.

There exists a well-established theory on stochastic evolution
equations in Hilbert spaces, see, for example, Da Prato and Zabcyck
\cite{DPZRosso}, that we shall apply in order to show that, for any
$n,m\in N$, equation \eqref{eq:EDSnm}, admits a unique solution
$X^{(n)}_m$. The precise statement concerning the well-posedness of
problems \eqref{eq:EDSnm} and \eqref{eq:EDSn} can be found in
\cite[Propositions 5.4 and 5.5]{AlDiPMa}. Moreover, it is possible
to prove an existence and uniqueness result for equation \eqref{e1}
in Section \ref{ssec:section1} through an approximating procedure on
finite dimensional spaces and with a Lipschitz continuous
nonlinearity. For more details concerning these results compare
\cite[Section 5]{AlDiPMa}, where the case of Gaussian noise is
treated. In the following we investigate the asymptotic properties
of the mild solution of \eqref{e1}. In particular, we are going to
prove that its unique solution admits a characterization in terms of
a stationary process $r$ and a process $v$ which vanishes at $t\to
\infty$. To this end, we proceed by splitting the solution
$X^{(n)}_m$  of the approximating problems into the sum of a
stationary process $r^{(n)}_m$ and a vanishing process $v^{(n)}_m$
satisfying suitable properties.
 Let us define the two sequences of processes $r^{(n)}_m$ and
$v^{(n)}_m$ respectively, in the sense of the definition \ref{ppp}
below for $r_m^{(n)}$ and $v_m^{(n)}(t)= X_m^{(n)}(t)-r_m^{(n)}(t)$,
as solutions of the equations:
\begin{align}
r^{(n)}_{m}(t)&=\int^{t}_{-\infty}
S(t-s)\Pi_{n}F_{m}\Pi_{n}(r^{(n)}_{m}(s))\dd s+{\bar{L}}^{\infty,n}_A(t) \label{eq:rnm}\\
v^{(n)}_{m}(t)&=S(t)\,x-\int^{0}_{-\infty}
e^{(t-s)A}\Pi_{n}F_{m}\Pi_{n}(r^{(n)}_{m}(s))\dd s- \int_{-\infty}^0 e^{(t-s)A}\Pi_n B\dd \bar{L}(s) \notag\\
&+\int^{t}_{0}e^{(t-s)A}[\Pi_{n}F_{m}(\Pi_{n}X^{(n)}_{m} (s))\dd
s-\Pi_{n}F_{m} (\Pi_{n}r^{(n)}_{m}(s))]\dd s, \label{eq:vnm}
\end{align}
where $S(t)$ is as in Section \ref{sec:OU}, $F_{m},\,\Pi_{n}$ are as
above and $\bar{L}^{\infty,n}_A(t)$, $t \in \mathbb{R}$, is defined
by
$\bar{L}^{\infty,n}_A(t)=\lim_{a\longrightarrow\,+\infty}\bar{L}^{a,n}_A(t)$,
with
\begin{align*}
L^{a,n}_A(t):=\int^{t}_{-a}S(t-s)\,\Pi_nB\,\dd\bar{L}(s),\,t \geq -a
\; , \; a \geq 0
\end{align*}
 We claim that
the random variable $L^{a,n}_A(t)$ is well-defined, for any $t \in
\mathbb{R}$. In fact, $L^{a,n}_A(t)$, $t \geq -a$, can be split into
the following sum:
\begin{align*}
  L^{\infty,n}_A(t)= \lim_{a \to +\infty} \int^{0}_{-a}S(t-s)\,\Pi_nB\,\dd\bar{L}(s)
     + \int_0^t S(t-s)\,\Pi_nB\,\dd\bar{L}(s)
\end{align*}
and the second term is easily seen to be  well-defined for any
$t\geq -a$. Concerning the first term we notice that, for any $a
\geq 0, t \geq -a$
\begin{multline*}
    \int^{0}_{-a}S(t-s)\,\Pi_nB\,\dd\bar{L}(s) = \sum_{k=1}^n \left(\int_{-a}^0 e^{-\lambda_k(t-s)}
   b_k \beta_k \dd \bar{L}^k(s) \right) e_k= \\ \sum_{k=1}^n e^{-\lambda_kt}\left(\int_{-a}^0 e^{\lambda_k s}
   b_k  \beta_k \dd \bar{L}^k(s)\right) e_k
   =  -\sum_{k=1}^n e^{-\lambda_kt}\left(\int_0^{a} e^{-\lambda_k s}
   b_k \beta_k \dd L^k(s) \right) e_k.
\end{multline*}
Now, taking into account Remark \ref{rem:LAinfty} in Subsection
\ref{ssec:invOU}, we see that
\begin{align*}
   \lim_{a \to +\infty} \sum_{k=1}^\infty \left(\int_0^{a} e^{-\lambda_k s}
   b_k \beta_k \dd L^k(s) \right) e_k
\end{align*}
is a well-defined $\cH$-valued random variable; hence
\begin{align*}
   \bar{L}_A^{\infty,n}(t)&= \int_{-\infty}^0 e^{(t-s)A } \Pi_nB\dd \bar{L}(s) \:,\: t \in \mathbb{R}
\end{align*}
and, taking $n \rightarrow +\infty$:
\begin{align*}
     L_A^\infty(t):= \int_{-\infty}^t e^{(t-s)A } B\dd \bar{L}(s)
\end{align*}
are well-defined too,for any $t \in \mathbb{R}$.

We give the following notion of solution for equation
\eqref{eq:rnm}.
\begin{Def}\label{ppp} An $\mathcal{F}_{t}$- adapted process $r^{(n)}_{m}$ is
said to be a mild solution to equation \eqref{eq:rnm} if it
satisfies the integral equation \eqref{eq:rnm} for any $t \in
\mathbb{R}.$
\end{Def}
\begin{theo} \label{BHRA} For any $n,m\in \mathbb{N},$ there exists a unique mild
solution $r^{(n)}_{m}$ to the equation \eqref{eq:rnm}, such that
\begin{equation}
\sup_{t \in \mathbb{R}}\mathbb{E}|r^{(n)}_{m}(t)|^{p}_{\cH} \leq
C_{p},
\end{equation}
for every $p\geq 2$ and for some positive constant
$C_{p}$(independent on $n$ and $m$). Further, $r^{(n)}_{m}$ is a
stationary process, that is, for every $h\in \mathbb{R}^{+}, k \in
\mathbb{N},$ any $-\infty  < t_{1}\ldots \leq t_{k} < +\infty$ and
any $A_{1},\ldots, A_{k} \in \mathcal{B}(\cal{H})$ we have
\[\mathbb{P}(r^{(n)}_{m}(t_{1}+h) \in A_{1},\ldots, r^{(n)}_{m}(t_{k}+h) \in A_k )=\mathbb{P}(r^{(n)}_{m}(t_{1}) \in A_{1}, \ldots , r^{(n)}_{m}(t_{n}) \in A_{k}).\]
\end{theo}
\begin{proof} 
Let us first prove the uniqueness: Assume that $(x(t))_{t \in
\mathbb{R}}$ and $(y(t))_{t \in \mathbb{R}}$ are solutions of
equation \eqref{eq:rnm}. Dissipativity of $A+\Pi_{n}F_{m}(\Pi_{n})$,
which follows from the assumptions, implies
\begin{align*}
\dd|x(t)-y(t)|^{2}&=\langle
A(x(t)-y(t))+\Pi_{n}F_{m}(\Pi_{n}x(t))-\Pi_{n}F_{m}(\Pi_{n}y(t)),x(t),x(t)-y(t)\rangle
\dd t\\
&\leq -2(\omega-\eta)|x(t)-y(t)|^{2}\dd t
\end{align*}
and, by using Gronwall's lemma, we deduce that for any $\xi > 0$ and
$t\geq -\xi$ the following inequality holds
$$|x(t)-y(t)|^{2}\leq |x(-\xi)-y(-\xi)|^{2}e^{-2(\omega-\eta)(t+\xi)}.$$
Letting $\xi\rightarrow +\infty$ we conclude that $x(t)=y(t)$ for
any $t \in \mathbb{R}.$\\
For the existence of a solution $r_m^{(n)}(t)$ to \eqref{eq:rnm}
let $r^{(n)}_{m}(t, -\xi),\,\xi > 0$ be  the unique solution of the
equation:
\begin{equation} \label{ABB}\left
\{\begin{array}{lll}
\dd r^{(n)}_{m}(t,-\xi)=Ar^{(n)}_{m}(t)dt+\Pi_{n}F_{m}\Pi_{n}(r^{(n)}_{m}(t,-\xi))dt+\Pi_{n}B\dd L(t), \quad t \geq -\xi \; ,\\
 r^{(n)}_{m}(-\xi,-\xi)=e^{-\xi A}x
\end{array}\right.
\end{equation}
Let us note that equation \eqref{ABB} has a unique solution being a
Cauchy problem with Lipschitz coefficients, see, e.g., \cite{PeZa}
for details. From the smoothness properties of the stochastic
convolution, one can assume that $r^{(n)}_{m}(\cdot; -\xi)$ admits a
${\cal B}$- c\`{a}dl\`{a}g version, which we still denote by
$r^{(n)}_{m}(\cdot; -\xi)$. Consequently, for any $t\in \mathbb{R}$,
we can define $r^{(n)}_{m}(t)$ as the limit of $r^{(n)}_{m}(t;-\xi)$
for $\xi \rightarrow +\infty$ and this turns out to be the solution
of equation \eqref{eq:rnm}.
Now we will prove that, for any $t\geq -\xi$, $\xi >0$ and $p\geq 2$
the following estimate holds
$$\sup_{-\xi \leq t \leq 0}\mathbb{E}\| r^{(n)}_{m}(t;-\xi)\|^{p}_{\cH}< C_{p} \:,$$
where $C_{p}$ is a positive constant independent on $n,m$ and $\xi$,
but possibly depending on $p$.  For simplicity, and without loss of
generality,  we consider the case $p=2a, a \in \mathbb{N}.$ We want
to apply It\^{o}'s formula to the processes
$|r^{(n)}_{m}(t,-\xi)|^{2a}_{\cH}.$ To this end, we recall the
expressions
for the first and second derivatives of the function $F(x):=|x|^{2a}.$\\
We have
\begin{align*}
&\nabla F(x)=2a|x|^{2(a-1)}x\\
&\frac{1}{2}Tr(Q\nabla
F^{2}(x))=aTr(Q)|x|^{2(a-1)}+(a-1)a|x|^{2(a-2)}|Bx|^{2}.
\end{align*}
Hence
\begin{align*}
\dd |r^{(n)}_{m}(t;-\xi)|^{2a}&=\langle
2a|r^{(n)}_{m}(t;-\xi)|^{2(a-1)}r^{(n)}_{m}(t;-\xi),dr^{(n)}_{m}(t;-\xi)\rangle \\
&+aTr(Q)|r^{(n)}_{m}(t;-\xi)|^{2(a-1)}dt+(a-1)a|r^{(n)}_{m}(t;-\xi)|^{2(a-2)}|Br^{(n)}_{m}(t;-\xi)|^{2}
\dd t.
\end{align*}
Now using the dissipativity of $A+F$, for sufficient small $\epsilon
>0,$ we get
\begin{align*}
\dd|r^{(n)}_{m}(t;-\xi)|^{2a}&=2a|r^{(n)}_{m}(t;-\xi)|^{2(a-1)}\langle
Ar^{(n)}_{m}(t;-\xi)+
\Pi_{n}F_{m}\Pi_{n}(r^{(n)}_{m}(t;-\xi)),r^{(n)}_{m}(t;-\xi)\rangle \dd t\\
&+2a|r^{(n)}_{m}(t;-\xi)|^{2(a-1)}
\langle \Pi_{n}B\dd L(t),r^{(n)}_{m}(t;-\xi)\rangle\\
&+aTr(Q)|r^{(n)}_{m}(t;-\xi)|^{2(a-1)}dt+(a-1)a|r^{(n)}_{m}(t;-\xi)|^{2(a-2)}|
Br^{(n)}_{m}(t;-\xi)|^{2}dt\\
&\leq
-2(\omega-\eta)|r^{(n)}_{m}(t;-\xi)|^{2a}_{\cH}+2a|r^{(n)}_{m}(t;-\xi)|^{2(a-1)}\langle
F(0),r^{(n)}_{m}(t;-\xi)\rangle \dd t\\
&+C_{a,Q}|r^{(n)}_{m}(t;-\xi)|^{2(a-1)}+2a|r^{(n)}_{m}(t;-\xi)|^{2(a-1)}\langle
\Pi_{n}B\dd L(t),r^{(n)}_{m}(t;-\xi)\rangle\\
&\leq
-2(\omega-\eta)|r^{(n)}_{m}(t;-\xi)|^{2a}_{\cH}+2a|r^{(n)}_{m}(t;-\xi)|^{2a-1}|F(0)|dt\\
&+C_{a,Q}|r^{(n)}_{m}(t;-\xi)|^{2(a-1)}+2a|r^{(n)}_{m}(t;-\xi)|^{2(a-1)}\langle
\Pi_{n}B\dd L(t),r^{(n)}_{m}(t;-\xi)\rangle\\
&\leq -2(\omega-\eta-\epsilon)|r^{(n)}_{m}(t;-\xi)|^{2a}_{\cH}\\
& +\frac{1}{\epsilon}C_{a,Q,F(0)} \dd
t+2a|r^{(n)}_{m}(t;-\xi)|^{2(a-1)} \langle \Pi_{n}B\dd L(t),
r^{(n)}_{m} (t;-\xi)\rangle.
\end{align*}
and, integrating over $[-\xi,t],t\geq -\xi$ we obtain:
\begin{align*}
|r^{(n)}_{m}(t;-\xi)|^{2a}\leq e^{-2a \xi \omega
}|x|^{2a}-2(\omega-\eta-\epsilon)\int^{t}_{-\xi}|r^{(n)}_{m}(s;-\xi)|^{2a}_{\cH}
\dd s\\
 +\int^{t}_{\xi}|r^{(n)}_{m}(s;-\xi)|^{2(a-1)} \langle
\Pi_{n} B\dd L(s),r^{(n)}_{m}(s;-\xi)\rangle+
\frac{1}{\epsilon} C_{a, Q, F(0)} (t+\xi).
\end{align*}
Notice that the term
\begin{align*}
 \int^{t}_{-\xi} \langle |r^{(n)}_{m}(s;-\xi)|^{2(a-1)} \Pi_{n} B\dd L(s), r^{(n)}_{m} (s;-\xi)\rangle
 \end{align*}
is a square integrable martingale with mean $0$, so that taking the
expectation of both members in the previous inequality we obtain
\begin{multline*}
\mathbb{E}|r^{(n)}_{m}(s;-\xi)|^{2a} \leq
e^{-2(\xi+t)\omega} |x|^{2a}+\frac{1}{\epsilon} C_{a, Q, F(0)}(\xi+t)-\\
2\left( \omega-\eta +\epsilon\right) \int^{t}_{-\xi}
\mathbb{E}|r^{(n)}_{m}(r;-\xi)|^{2} \dd s\:.
\end{multline*}
Then applying  Gronwall's lemma we get
\begin{eqnarray}\mathbb{E}|r^{(n)}_{m}(t;-\xi)|^{2a} &\leq& [e^{-2a(\xi
+t)\omega}|x|^{2a}+\frac{1}{\epsilon}C_{a,Q,F(0)}(\xi+t)]e^{-2a(\omega-\eta-\epsilon)\xi}\nonumber\\
&\leq& C,
\end{eqnarray}
where $C$ is a suitable constant independent on $m,n $ and $\xi$.

Moreover, in a similar way we can prove that, for any fixed $t \in
\mathbb{R}$, the sequence $\{ r^{(n)}_{m} (t;-\xi)\}_{-\xi \leq t}$
is a Cauchy sequence, uniformly on $t$.

Now let $0\leq \gamma \leq \xi .$ We need to estimate the norm:
\begin{align*}
 \sup_{t\geq -\gamma} \mathbb{E}|r^{(n)}_{m}(t;-\xi)-r^{(n)}_{m}(t;-\gamma)|^{2a}.
 \end{align*}
To this end we notice that the process $y_{\xi,\gamma}
(t):=r^{(n)}_{m}(t;-\xi)-r^{(n)}_{m}(t;-\gamma)$ can be written as:
\begin{equation} \label{ARBH}
\begin{aligned}
y_{\xi,\gamma}(t)&=(e^{-\xi a}-e^{-\gamma A} )x +\int^{-\gamma}_{-\xi}e^{(-\gamma-s)A}\Pi_{n}F_{m}(\Pi_{n}r^{(n)}_{m}(s;-\xi))ds\\
&+ \int^{-\gamma}_{-\xi}e^{(-\gamma-s)A}B\dd L(s)\nonumber\\&
+\int^{t}_{-\gamma}e^{(t-s)A}[\Pi_{n}F_{m}
(\Pi_{n}r^{(n)}_{m}(s;-\xi))-\Pi_{n}F_{m}
(\Pi_{n}r^{(n)}_{m}(s;-\gamma))]ds.
\end{aligned}
\end{equation}

Moreover, the stochastic differential of $y_{\xi,\gamma}(t)$ is
given by
\[\dd y_{ \xi, \gamma}(t)=Ay_{\xi , \gamma }dt+[F_{n,m}(r^{(n)}_{m}(t;-\xi))-F_{n,m}(r^{(n)}_{m}(t;-\gamma))]dt,\]
since the three terms in \eqref{ARBH} do not depend on $t$. By
applying It\^o's formula to $|y_{\xi,\gamma}|^{2a}$ and reasoning as
above we get
\begin{align*}
\mathbb{E}|y_{\xi, \gamma }(t)|^{2a}\leq |e^{-\xi
x}|^{2a}+\mathbb{E}|y_{\xi,\gamma}(-\gamma)|^{2a}-2(\omega -\eta
)\int^{t}_{-\xi} \mathbb{E}|y_{\xi, \gamma }(r)|^{2a} \dd s.
\end{align*}

Gronwall's lemma then implies
\begin{eqnarray}\sup_{t\geq -\gamma}\mathbb{E}|y_{\xi, \gamma
}(t)|^{2a}&\leq&(\mathbb{E}|y_{\xi,\gamma}(-\gamma)|^{2a}+|e^{-\xi
Ax}|^{2a})e^{-2(\omega-\eta )(\xi+t)} \nonumber\\&\leq& C,
\end{eqnarray}
where $C$ is a constant independent on $\xi$. We then conclude that
for any $t \in \mathbb{R}$  the limit $r^{(n)}_{m}(t):=\lim_{\xi
\rightarrow \infty}r^{(n)}_{m}(t;-\xi)$ exists in
$L^{2}(\Omega,\mathbb{P};{\cal B})$ and
moreover,\\
\[\displaystyle \sup_{t\geq -\xi}|\mathbb{E}r^{(n)}_{m}(t)|^{2a} \leq C.\]
In addition, by the condition at $t = -\xi$ in \eqref{ABB}, we
deduce that $\lim_{t\rightarrow -\infty}r^{(n)}_{m}(t)=0.$

Finally, for any $n,m \in \mathbb{N}$ we are going to show that  the
process $r^{(n)}_{m}$ is stationary. In order to prove this
statement, we adapt to our case the argument given in \cite{Marc2}.
In particular, we introduce the following Picard iteration:
\begin{equation}\label{hhh}\left\{\begin{array}{lll}
r^{(n,0)}_{m} (t)=x\\
r^{(n,k+1)}_{m} (t)=\int^{t}_{-\infty}
e^{(t-s)A}\Pi_{n}F_{m}(\Pi_{n}r^{(n,k)}_{m}(s))\dd
s+L_{A,-\infty}(t).
\end{array}\right.
\end{equation}

We notice that  the limit $\lim_{k\rightarrow
\infty}r^{(n,k)}(t)=\tilde{r}^{(n)}_{m}(t)$ exists (see, e.g,
\cite{DPZVerde}) and it is a stationary process. The crucial point
is that $\tilde{r}^{(n)}_{m}$ and $r^{(n)}_{m}$ coincide. In fact,
if we pass to the limit in \eqref{hhh}, we see that $r^{(n)}_{m}$
solves equation \eqref{eq:rnm} so that, by uniqueness,
$\tilde{r}^{(n)}_{m}\equiv r^{(n)}_{m}.$ Consequently, $r^{(n)}_{m}$
is stationary. This completes the proof of theorem \eqref{BHRA}.
\end{proof}

Taking into account our assumptions on the double sided convolution
process $\bar{L}_A(t)$ we will discuss  existence and uniqueness of
a mild solution for the equation for the process $v^{(n)}_{m}$ and
we will show that it vanishes, in a suitable sense, as $t\rightarrow
+\infty.$

\begin{prop} Under the assumptions given in Theorem \ref{t1} and in Theorem \ref{t2} we have that
for any $n,m \in \mathbb{N}$ there exists a unique mild solution
$(v^{(n)}_{m}(t))_{t\geq 0}$ of equation \eqref{eq:vnm}. Moreover,
for any $p\geq 1,$ we have the following bound:
\begin{align*}
 \sup_{t\geq 0} \mathbb{E}|v^{(n)}_{m}
(t)|^{p}\leq C_{p},
\end{align*}
where $C_{p}$ is a positive constant independent of $n$ and $m$. In
addition, we have the following limit
\begin{align*}
\lim_{t\rightarrow +\infty} \mathbb{E}|v^{(n)}_{m} (t)|^{p}=0\; , \;
\textit{ for any } p \geq 1\; .
\end{align*}
\end{prop}
\begin{proof} The  existence and uniqueness of a mild solution for the process $(v^{(n)}_{m})$ follow straightforward by results in finite dimensions,
 see.eg, \cite{APD}.

Without loss of generality we can assume that $p=2a$ for $a \in
\,\N$. By the dissipativity of the mapping $\Pi_{n}F_{m}\Pi_{n}$ and
the fact that $A\leq -\omega$, we get
\begin{align*}
d|v^{(n)}_{m}(t)|^{2a}&=\dd|X^{(n)}_{m}(t)-r^{(n)}_{m}(t)|^{2a}\\
&=2a\langle
Av^{(n)}_{m}(t)+\Pi_{n}F_{m}(\Pi_{n}X^{(n)}_{m}(t))-\Pi_{n}F_{m}(\Pi_{n}r^{(n)}_{m}(t)),v^{(n)}_{m}(t)\rangle
|v^{(n)}_{m}(t)|^{2a-2} \, \dd t\\
& -2a\omega |v^{(n)}_{m}(t)|^{2a}+2a\eta |v^{(n)}_{m} (t)|^{2a},
\end{align*}
so that integrating on $[0,t]$ and applying Gronwall's lemma, we
obtain
\begin{align*} \sup_{t \in [0,T]}{\mathbb{E}}|v^{(n)}_{m}(t)|^{2a}
&\leq e^{-(\omega-\eta) T}{\mathbb{E}}|v^{(n)}_{m}(0)|^{2a}
\nonumber\\&\leq C_{a}e^{-(\omega-\eta) T}
\left[|x|^{2a}_{\cH}+{\mathbb{E}}\int^{0}_{-\infty} |e^{-sA}
\Pi_{n}F_{m}(\Pi_{n}r^{(n)}(s))|^{2a}_{\cH} \dd s+ \right.
\\& \quad  +
\left. \left| \int^{0}_{-\infty} e^{-sA} \Pi_{n} B\dd
L(s)\right|^{2a}
\right]\\
&\leq C_{a}e^{-(\omega-\eta)T} \left[K_a+{\sup}_{t\geq0}\,\mathbb{E}
\left( \left| L_A(t) \right|_{\mathcal{H}} + \left|
F\left(L_A(t)\right) \right|_{\mathcal{H}} \right)\,\right] \:,
\end{align*}
where $C_{a}, K_a$ are positive constants depending only on $a$. Now
from this inequality and the assumptions in Theorem \ref{t2}, we
deduce that
\begin{equation}
 \sup_{t \in [0,T]}
\mathbb{E}|v^{(n)}_{m} (t)|^{2a} \leq C_{a} e^{-\omega T} (K_a+C),
\end{equation}
with $\omega-\eta > 0,$ the result now follow by letting
$T\rightarrow +\infty$.
\end{proof}
The next result states that $r^{(n)}_{m} $ and $v^{(n)}_{m}$
converge respectively to stochastic processes $r$ and $v$ in
$L^{p}(\Omega, C([0,T];{\cal B})), p \geq 1, T>0$, where ${\cal B}$
is as in Sect.3, moreover it also shows additional properties of $r,
v$.

\begin{prop} There exist a stationary process $r$ and $a$
process $v$ in $L^{p}(\Omega;C([0,T];{\cal B}))$ such that
\begin{align*}
 &\lim_{n,m \rightarrow \infty }r^{(n)}_{m} (t)=r(t) \\
& \lim_{n,m \rightarrow \infty }v^{(n)}_{m} (t)=v(t).
\end{align*}
Further, for any $p\geq 1$, $\lim_{t\rightarrow +\infty}
\mathbb{E}|v(t)|^{p}=0$.
\end{prop}
\begin{proof} Again without loss of generality, we assume that
$p=2a,a \in \mathbb{N}.$ For the convergence of the sequence
$r^{(n)}_{m}$ and $v^{(n)}_{m}$ the proof is by contradiction.
Assume that there exists $\varepsilon > 0$ such that, for all $m , n
\in \N$ :
\[\displaystyle \sup_{k,k'>n,j,j'> m}\mathbb{E}|r^{(k)}_{j}(t)-r^{(k')}_{j'} (t)|^{2a}_{\cH} > 2 \varepsilon .\]
Since the difference of two stationary processes is stationary, the
expression on the left hand side is independent of time $t$. By
choosing $t$ large enough, thus making $\mathbb{E}|v^{(k)}_{j}
(t)|^{2a}$ and $\mathbb{E}|v^{(k')}_{j'} (t)|^{2a}$ sufficiently
small, it is easy to show that
\[\displaystyle \sup_{t\geq 0\, k,} \sup_{k'>n,j,j'>m} \mathbb{E}|X^{(k)}_{j}(t)-X^{(k')}_{j'}(t)|^{2a}> \varepsilon.\]
But this contradicts the fact that $\lim_{n,m \rightarrow \infty}
X^{(n)}_{m}(t)$ exists. As a consequence of the convergence of the
sequence $r^{(n)}_{m}$ and $X^{(n)}_{m}$
we obtain the convergence of $v^{(n)}_{m}.$\\
Now let us show that:
\begin{equation}
 \lim_{t\rightarrow +\infty}
\mathbb{E}|v(t)|^{2a}=0,
\end{equation}
where $v(t):= X(t)-r(t)$, $t \geq 0$.
 For this we have that:
\begin{multline*}
\mathbb{E}|v(t)|^{2a}=\mathbb{E}|X(t)-r(t)|^{2a}\leq
c_{a}\mathbb{E}|X(t)-X^{(n)}_{m}(t)|^{2a}\\
+c_{a}\mathbb{E}|X^{(n)}_{m}(t)-r^{(n)}_{m}(t)|^{2a}+c_{a}\mathbb{E}|r^{(n)}_{m}
(t)-r(t)|^{2a},
\end{multline*}
 for some strictly positive constant $c_{a}$
depending only on $a$. If $n,m$ are large enough, then the first and
third terms are less than $\varepsilon/c_{a},$ uniformly in $t\geq
1$. The second term is less than $ \varepsilon/c_{a},$ for
$t>T(\varepsilon)$, for some $T(\varepsilon)$ independent of $n, m$
for all $n, m$ large. Combining the previous three terms we have
shown that $\mathbb{E}|v(t)|^{2a} < \varepsilon $ for sufficient
large positive $t$.
\end{proof}
From Theorem \ref{t2} we had alrea\dd y the existence of the
invariant measure for the process $X$. We shall now prove that is
given by the law of the stationary process $r$:
\begin{theo} The invariant measure for the process $(X(t))_{t\geq
0}$, is given by the law $\mathcal{L}(r(t))$ of the stationary
process $r$.
\end{theo}
\begin{proof} It suffices to prove that the law  $\mathcal{L}(r(t))$
of $r(t)$ is an invariant measure for $X$, which is implied by the
stationarity of $r(t)$. Moreover, exploiting the uniqueness of
invariant measures for $X$, see Theorem \ref{t2}, we have that
$\mathcal{L}(r(t))$ is the unique invariant measure for $X$.
\end{proof}

\section*{Acknowledgments}
\noindent  This work was supported by King Fahd University of
Petroleum and Minerals under the project $\sharp\,IN121060$. The
authors gratefully acknowledges this support.\\We thank Stefano
Bonaccorsi and
Luciano Tubaro at the University of Trento for many stimulating discussions.\\
The authors would also like to gratefully acknowledge the great
hospitality of various institutions. In particular for the first
author CIRM and the Mathematics Department of the University of
Trento; for him and the fourth author  King Fahd University of
Petroleum and Minerals at Dhahran; for the second, third  and fourth
authors IAM and HCM at the University of Bonn, Germany.

\medskip

\begin{flushleft}
\footnotesize
{\it S. Albeverio}  \\
 Dept. Appl. Mathematics, University of Bonn,\\
HCM; KFUPM, BiBoS, IZKS\\
\medskip
{\it L. Di Persio}\\
University of Verona, Department of Computer Science, \\
   Italia \\
\medskip
{\it E. Mastrogiacomo}\\
Universit\'a degli studi di Milano Bicocca, Dipartimento di Statistica e Metodi Quantitativi \\
   Piazza Ateneo Nuovo, 1 20126 Milano
\\
\medskip
{\it B. Smii} \\
Dept. Mathematics, King Fahd University of Petroleum and Minerals, \\
Dhahran 31261, Saudi Arabia\\

\medskip
 E-mail: albeverio@uni-bonn.de \\
\hspace{1,1cm} luca.dipersio@univr.it\\
\hspace{1,1cm} elisa.mastrogiacomo@polimi.it\\
\hspace{1,1cm}boubaker@kfupm.edu.sa

 \end{flushleft}

\begin{thebibliography}{99}
\bibliographystyle{alpha}
\footnotesize
\bibitem{Alb} {\sc Albeverio. S.}  Theory of Dirichlet forms and
applications. {\it Lectures on probability theory and statistics}
(Saint-Flour, 2000), 1-106, Lecture Notes in Math., 1816, Springer,
Berlin, 2003.

\bibitem{A}{\sc
    Albeverio. S.}
     Wiener and Feynman-path integrals and their applications,
  {\it Proceedings of the {N}orbert {W}iener {C}entenary {C}ongress,
              (1994), {E}ast {L}ansing, {MI},
   Proc. Sympos. Appl. Math.},
    {\bf 52},
 Amer. Math. Soc., Providence, RI, 1997, pp. 153--194.

    \bibitem{AlBogRoc}
    {\sc Albeverio. S and Bogachev, V. and R{\"o}ckner, M.}, On uniqueness of invariant measures for finite and
              infinite-dimensional diffusions, {\it Comm. Pure Appl. Math.}, (1999)  {\bf 52}, No. 3.

\bibitem{AC}
    {\sc Albeverio, S and Cebulla. C.}
    Synchronizability of stochastic network ensembles in a model of interacting \dd ynamical units,
   {\it  Physica A Stat. Mech. Appl},
     {\bf 386},
     pp. 503--512.

\bibitem{AlDiP} { \sc Albeverio. S and  Di Persio. L.} {\it Some stochastic dynamical models in neurobiology: recent developments.}  European Communications in Mathematical and Theoretical Biology, No.14

\bibitem{ADPM}{
    \sc Albeverio. S, Di Persio L. and Mastrogiacomo. E.}
  {\it   Invariant measures for stochastic differential equations on
networks}, Proocedings of Symp. in Pure Mathematics, vol. {\bf 87},
pp. 1--34, AMS, ed. H. Holden et all., AMS (2013).

\bibitem{AlDiPMa} {\sc Albeverio. S, Di Persio L. and Mastrogiacomo. E.} {\it Small noise asymptotic expansions for
stochastic PDE's I. The case of a dissipative polynomially bounded
nonlinearity. } Tohoku. Math. J., {\bf  63} (2011), pp. 877--898.

\bibitem{AFP}{
   \sc Albeverio. S, Fatalov. V and Piterbarg. V.
              I.},
     {\it Asymptotic behavior of the sample mean of a function of the
              {W}iener process and the {M}acdonald function},
   J. Math. Sci. Univ. Tokyo,
    {\bf 16},
     (2009),
    pp. 55--93.

\bibitem{AFER}
    {\sc Albeverio. S and Ferrario. B.} {\it Some methods of
    infinite dimensional analysis in hydrodynamic: recent progress
    and prospects. Lecture Notes in Math.  V 1942.} Springer,
    Berlin, 1-50 (2008).

\bibitem{AFS}
    {\sc Albeverio. S, Flandoli. F and Sinai, Y. G.}
   S{PDE} in hydro\dd ynamic: recent progress and prospects,
   {\it Lectures given at the C.I.M.E. Summer School held in Cetraro,
              August 29--September 3, 2005,
              Edited by G. Da Prato and M. R\"ockner}.
    Lecture Notes in Mathematics,
    vol. 1942,
 Springer-Verlag,
   Berlin,
     (2008).

\bibitem{ALGYo}{\sc Albeverio, S.,  Gottschalk, H. and  Yoshida, M. W.}  Systems of classical particles in the grand canonical ensemble, scaling
 limits and quantum field theory.
{\it  Rev. Math. Phys. } {\bf 17}  (2005),  no. 2, 175--226.

\bibitem{AKaMi} Albeverio, S., Kawabi, H. Mihalache, S. and R\"ochner, M., in preparation

\bibitem{ALEBOU} {\sc Albeverio. S, Mastrogiacomo. E and Smii.
B.\,}{\it Small noise asymptotic expansions for stochastic PDE's
driven by dissipative nonlinearity and L\'evy noise.}Stoch. Process.
Appl. 123 (2013), 2084--2109.







\bibitem{AGoWu} {\sc  Albeverio. S, Gottschalk. H and J-L. Wu. } {\it Convoluted Generalized White noise, Schwinger Functions and their Analytic continuation to
 Wightman Functions.} Rev. Math. Phys, Vol. 8, No. 6, 763-817, (1996).
 \bibitem{AlGYo} {\sc Albeverio. S,  Gottschalk. H and Yoshida. M.W. } {\it System of classical particles in the Grand canonical ensemble, scaling limits and
 quantum field theory.}  Rev. Math. Phys, Vol. 17, No. 02, 175-226, (2005).

 \bibitem{AHK1} {\sc Albeverio, S. and Hoegh-Krohn, R. } Quasi invariant measures, symmetric diffusion processes
and quantum fields {\it Les M\'ethodes Math\'matiques de la
Th\'eorie Quantique des Champs}Colloques Internationaux du Centre
Nat. Rech. Sci. Marseille, 23-27 juin 1975, C.N.R.S. 1976, pp.
11--59


\bibitem{AHK2} {\sc  Albeverio, S. and Hoegh-Krohn, R.} Dirichlet forms and diffusion processes on rigged Hilbert
spaces.
 {\it Z. Wahr. Theor. Verw. Geb }{\bf 40}  (1977),  1--57.



 \bibitem{AHKS} {\sc Albeverio, S.,  Hida, T., Potthoff, J.,  R\"ockner, M. and Streit, L. } Dirichlet forms in terms of white noise analysis. I. Construction and
 QFT examples.
 {\it Rev. Math. Phys. }{\bf 1 } (1989),  no. 2-3, 291--312.

\bibitem{ALKR}{\sc Albeverio. S, Kawabi. H and R$\ddot{o}$ckner. M.}
{\it Strong uniqueness for both Dirichlet operators and stochastic
\dd ynamics to Gibbs measures on a path space with exponential
interactions.} J. Funct. Anal. 262 (2012), no. 2, 602-638.
\bibitem{AL}{
    \sc Albeverio. S and Liang. S.}
   Asymptotic expansions for the {L}aplace approximations of sums
              of {B}anach space-valued random variables,
 {\it  Ann. Probab.}
   {\bf 33},
    (2005),
     pp. 300--336.

\bibitem{AlMaLy}
     {\sc Albeverio. S, Lytvynov. E and Mahnig.  A.},
   {\it  A model of the term structure of interest rates based on
              L\'evy fields. Stochastic Process. Appl. }114 (2004), no.
              251-263.

\bibitem{ABRM} {\sc Albeverio, S., Mandrekar, V. and R\"udiger, B.}
{\it Existence of mild solutions for stochastic differential
equations and semilinear equations with non-Gaussian L\'evy noise. }
Stochastic Process. Appl.{\bf 119} (2009), no. 3, pp.835--863.

\bibitem{AMa}{\sc Albeverio. S and Mazzucchi. S.},
     The trace formula for the heat semigroup with polynomial potential,
  {\it Proc. Seminar Stochastic Analysis, Random Fields and Applications VI, Ascona 2008},
  Birkh\"auser,
  Basel,
   (2011),
   pp. \ 3--22,
  Edited by  R. Delang, M. Dozzi, F. Russo.


\bibitem{AlRo}
     {\sc Albeverio. S and R\"ockner. M.}
     Stochastic differential equations in infinite dimensions:
              solutions via {D}irichlet forms,
   {\it Probab. Theory Related Fields},
    {\bf 89},
     (1991),
     pp. \ 347--386.



\bibitem{ARoSk}
    {\sc Albeverio. S, R\"ockle. H and Steblovskaya. V.},
     Asymptotic expansions for {O}rnstein-{U}hlenbeck semigroups
              perturbed by potentials over {B}anach spaces,
{\it Stochastics Stochastics Rep.},
{\bf 69},
     (2000),
     pp. 195--238.

\bibitem{ABR1} {\sc Albeverio. S and R\"udiger B.}, Stochastic integrals and the L\'evy-It$\hat{o}$ decomposition theorem on separable Banach spaces.
 {\it Stoch. Anal. Appl.}{\bf  23} (2005), no. 2, 217--253.

\bibitem{ABRW} {\sc Albeverio. S, R\"udiger. B and Wu, J.L.},
{\it Invariant measures and Symmetry property of L\'evy type
operators.} Pot. Ana. 13 (2000), 147--168.

\bibitem{ASK}
    {\sc Albeverio. S and Steblovskaya. V.}
    Asymptotics of infinite-dimensional integrals with respect to
              smooth measures. {I},
  {\it Infin. Dimens. Anal. Quantum Probab. Relat. Top.},
     {\bf 2},
      (1999),
     pp. \ 529--556.
\bibitem{AlWuZh}
    {\sc Albeverio. S, Wu. J-L and Zhang. T.S.} {\it Parabolic SPDEs driven by Poisson white noise.
    } Stochastic Proc. Appl. 74, 21-36 (1998).

    \bibitem{A-Cufaro} {\sc Andrisani A. and  Cufaro Petroni, N.}
Markov processes and generalized Schr\"odinger equations {\it J.
Math. Phys.}{\bf 52} (2011).



\bibitem{APD} {\sc Applebaum, D.} {\it L\'evy processes and
stochastic calculus}. 2nd ed., Cambridge U.P, (2009)

\bibitem{AW} {\sc Applebaum, D.  and  Wu., J.L .},  Stochastic partial differential equations driven
by L\'evy space time white noise,{\it Random Ops. and Stochastic
equations.} {\bf 8}, 245-61 (2000).



\bibitem{BaNie} {\sc Barndorff-Nielsen, E. and Basse-O'Connor, A.}
{\it Quasi Ornstein Uhlenbeck processes} Bernoulli Volume {\bf 17},
Number 3 (2011), pp. 916--941.

\bibitem{BoMZ}  {\sc  Bonaccorsi, S, Marinelli, C. and Ziglio, G.}
 Stochastic Fitz-Hugh Nagumo equations on networks with
impulsive noise, {\it Electr. J. Prob.}{\bf 13}, 1362--1379 (2008).

\bibitem{BoMa}
  {\sc  Bonaccorsi. S and Mastrogiacomo. E.}
    Analysis of the stochastic {F}itz{H}ugh-{N}agumo system,
{\it  Infin. Dimens. Anal. Quantum Probab. Relat. Top.},
   {\bf 11},
      (2008)
     pp. \ 427--446.

\bibitem{BRHA} {\sc Brze\'zniak, Z. and Hausenblas, E.},
Uniqueness in law of the It\^o integral with respect to L\'evy
noise, in {\it Seminar Stoch. Anal., Random Fields and Appl.}, VI,
Birkhauser, Basel (2011), pp. 37--57.

\bibitem{BrzPe} {\sc Brzez\'niak, Z. and  Peszat, S.} {\it  Space-time continuous solutions to SPDE's driven by a homogeneous
 Wiener process.}
 Studia Math.  {\bf 137}  (1999),  no. 3, 261--299.

\bibitem{CaMu} {\sc Cardanobile. S and Mugnolo. D. }
    Analysis of a
     FitzHugh-Nagumo-Rall model of a neuronal network,
  {\it  Math. Methods
   Appl. Sci.}  \textbf{30} (2007), no.~18, pp. 2281--2308.

 \bibitem{Carm} {\sc Carmona. R. A and Tehranchi. M.R.}  {\it Interest Rate Models: an Infinite Dimensional Stochastic Analysis
    Perspective}, Springer Finance 2006.


\bibitem{Cerrai99}
 {\sc   Cerrai. S.}
   Differentiability of {M}arkov semigroups for stochastic
              reaction-diffusion equations and applications to control,
  {\it Stochastic Process. Appl.},
   {\bf 83}
      (1999),
   no.~1,
    pp. 15--37.

\bibitem{CeF}
    {\sc Cerrai. S and Freidlin. M.}
     Smoluchowski-{K}ramers approximation for a general class of
              {SPDE}s,
 {\it  J. Evol. Equ.},
   {\bf 6}
     (2006),
no.~4,
     pp. 657--689.

\bibitem{Cou} {\sc Courr\'ege, Ph.}, Sur la forme int\'egro-diff\'erentielle des op\'erateurs de $C^\infty_k(\mathbb{R}^n)$ dans $C(\mathbb{R}^n)$ satisfaisant au principe du maximum", S\'em. Th\'eorie du potentiel (1965/66) Expos\'e 2


\bibitem{CuPe} {\sc Cufaro Petroni, N.}
L\'evy-Schr\"odinger wave packets {\it J. Phys. A: Math. Theor.}{\bf
44} (2011)


\bibitem{DaPraDeb} {\sc Da Prato, G. and  Debussche, A.} {\it Strong solutions to the stochastic quantization equations.}
 Ann. Probab.  {\bf 31}  (2003),  no. 4, 1900--1916.

\bibitem{DaPTu}
     {\sc Da Prato. G. and Tubaro. L.}
    Self-adjointness of some infinite-dimensional elliptic
              operators and application to stochastic quantization,
   {\it Probab. Theory Related Fields},
    {\bf 118}
     (2000),
    no.~1,
     pp. 131--145.

\bibitem{DPZVerde}
    {\sc  Da Prato. G and Zabczyk. J.}
    Ergodicity for infinite-dimensional systems,
    {\it London Mathematical Society Lecture Note Series},
    vol. {\bf 229},
 Cambridge University Press,
   Cambridge,
   (1996).

\bibitem{DPZRosso}
    {\sc Da Prato. G. and Zabczyk. J.}
     Stochastic equations in infinite dimensions,
    {\it Encyclopedia of Mathematics and its Applications},
    vol. {\bf 44},
 Cambridge University Press,
   Cambridge,
     (1992),
     pp.  xviii+454.

\bibitem{DaMue} {\sc Dalang. R.C and Mueller. C} {\it Some non-linear SPDE's that are second order in
time.} Electronic J. Probab., 8. 1, 1-21, (2003).

\bibitem{DeZe}
   {\sc Dembo. A and Zeitouni. O. }
   Large
  Deviations Techniques and Applications,
  {\it Applications of Mathematics} vol. {\bf 38},
  Second edition,
  Springer-Verlag, New York, (1998), pp. xvi+396.


\bibitem{DeStr}
   {\sc  Deuschel, J.D. and Stroock, D. W.}
    Large deviations,
   {\it Pure and Applied Mathematics},
     {\bf 137},
Academic Press Inc.,
   Boston, MA,
      (1989).

  \bibitem{Dy}
  {\sc Dynkin, E. B. } Diffusions, superdiffusions and partial differential equations.
{\it American Mathematical Society Colloquium Publications}, {\bf
50}. American Mathematical Society, Providence, RI,  2002. xii+236
pp. ISBN: 0-8218-3174-7

\bibitem{Elg}
    {\sc Lehnertz. K, Arnhold. J, Grassberger. P. and Elger. C.E.}
     Chaos in Brain ?,
 World Scientific,
   Singapore,
      (2000).

\bibitem{OZS} {\sc Fehmi, O. and Schmidt, T.}, Credit risk with
infinite dimensional L\'evy processes., {\it Stat. and Dec.} 23, pp.
281--299 (2005).

\bibitem{LuT}
    {\sc Forster. B, L{\"u}tkebohmert, E and Teichmann,
              J.}
     Absolutely continuous laws of jump-diffusions in finite and
              infinite dimensions with applications to mathematical finance,
   {\it SIAM J. Math. Anal.},
   {\bf 40},
     (2008/09),
    no.\,{5},
    pp. 2132--2153.

\bibitem{FuOTa} {\sc Fukushima. M, Oshima. Y and Takeda. M.} {\it
Dirichlet forms and symmetric Markov processes.} Second revised and
extended edition de Gruter Studies in Mathematics, 19. Walter de
Gruyter and Co., Bertlin, (2011).

\bibitem{GaMa} {\sc Gawarecki. L and Mandrekar. V.} {\it Stochastic Differential Equations in Infinite Dimensions: with Applications to Stochastic Partial
 Differential Equations.} Springer 2010.

\bibitem{GST} {\sc Gottschalk, H., Smii, B. and Thaler. H.},
The Feynman graph representation of general convolution semigroups
and its applications to L\'evy statistics. {\it  J. Bern. Soc},
{\bf14} (2), pp. 322--351,  (2008).

\bibitem{GoSm}{\sc Gottschalk, H. and Smii, B.} How to determine the law of the solution to a SPDE driven by a L\'evy space-time noise,  {\it J. Math. Phys.}
{\bf 43} pp. 1--22, (2007).

\bibitem{HABU} {\sc Hausenblas, E.}  Burkholder-Davis-Gun\dd y type
inequalities of the It\^o stochastic integral with respect to L\'evy
noise on Banach spaces, {\tt arXiv:0902.2114 [math.PR]}, (2009).

\bibitem{Holden-et-all} {\sc Holden, H., Oksendal, B., Ub\"oe, J and Zhang. T.}
 {\it Stochastic partial differential equations. A modeling, white noise
functional approach.} Second edition. Universitext. Springer,
NewYork, 2010.

\bibitem{IK07}
{\sc Inahama. Y and Kawabi. H.}
  Asymptotic expansions for the Laplace approximations for It\^o functionals of Brownian rough paths,
 {\it J. Funct. Anal.} {\bf 243} (2007), pp. \ 270--322.

\bibitem{InKa}
    {\sc Inahama, Y. and Kawabi, H.}
      On the {L}aplace-type asymptotics and the stochastic {T}aylor
              expansion for {I}t\^o functionals of {B}rownian rough paths,
 {\it Proceedings of {RIMS} {W}orkshop on {S}tochastic {A}nalysis
              and {A}pplications},
    RIMS K\^oky\^uroku Bessatsu, B6,
 Res. Inst. Math. Sci. (RIMS), Kyoto,
      (2008)
 pp. 139--152.

\bibitem{IkWa}
    {\sc Ikeda, N. and Watanabe, S.}
      Stochastic differential equations and diffusion processes,
    {\it North-Holland Mathematical Library},
    {\bf 24}, (1989)
  Second edition,
 North-Holland Publishing Co.,
  Amsterdam.




\bibitem{Jac}{\sc Jacob, N.}, Pseudo differential operators and {M}arkov processes. {V}ol.
              {I}, Fourier analysis and semigroups, Imperial College Press, London (2001).


\bibitem{Jac0}{\sc Jacob, N.}, Characteristic functions and symbols in the theory of {F}eller processes,
 {\it Potential Anal.}, Vol.{\bf 8}, (1998), No.1.



\bibitem{Jacob-Sch} {\sc Jacob, N.,  Schilling, R. L.}  L\'evy-type processes and pseudodifferential operators.
 {\it L\'evy processes},
 139--168, Birkh\"auser Boston, Boston, MA,  2001.

\bibitem{JoMi}
    {\sc Jona-Lasinio, G. and Mitter, P. K.}
     Large deviation estimates in the stochastic quantization of
              $\phi^4_2$,
   {\it Comm. Math. Phys.},
    {\bf 130},
      (1990),
   no.\,1,
    pp. 111--121.

\bibitem{JLM}
    {\sc Jona-Lasinio. G and Mitter. P. K.}
     On the stochastic quantization of field theory,
 {\it Comm. Math. Phys.},
 {\bf 101},
      (1985),
    no.\,3,
    pp. 409--436.

 \bibitem{KALLO} {\sc Kallenberg. O.} {\it Foundations of modern probability.} Springer (1997).

\bibitem{Kall} {\sc Kallianpur, G. and Wolpert, R. L.} Weak
convergence of stochastic neuronal models. {\it In stochastic
methods in biology } (Nagoya, 1985), {\bf 70 } of Lecture Notes in
Biomathematics, pages 116--145. Springer, Berlin, 1987.

\bibitem{KES} {\sc Keener, J. and Sneyd, J.} {\it Mathematical
physiology}. Second edition. Interdisciplinary applied Mathematics,
8/I. Springer, New York, 2009.

\bibitem{KX} {\sc Kallianpur. G and Xiong. J} {\it Stochastic Differential Equations on Infinite Dimensional
Spaces.} IMS Lecture notes-monograph series {\bf 26}, 1995.

\bibitem{KolmogorovFomin}
    {\sc Kolmogorov. A. N. and Fomin. S. V.}
     Elements of the theory of functions and functional analysis.
              {V}ol. 2: {M}easure. {T}he {L}ebesgue integral. {H}ilbert
              space,
   {\it Translated from the first (1960) Russian ed. by H. Kamel
              and H. Komm},
 Graylock Press,
   Albany, N.Y.,
      (1961),
     pp. ix+128.


\bibitem{Lal}
    {\sc Ladas. G. E and Lakshmikantham. V.}
     Differential equations in abstract spaces,
     Mathematics in Science and Engineering, Vol. {\bf 85},
 Academic Press,
   New York,
     (1972).

\bibitem{LOK} {\sc L\"okka. A, Oksendal. B and Proske. F.}
{\it Stochastic partial differential equations driven by L\'evy
space-time white noise.} Ann. Appl. Prob. 14, 1506-1528 (2004).


\bibitem{Le24} {\sc L\'evy, P.} Théorie des erreurs. La loi de Gauss et les lois exceptionelles.
Bull. Soc. Math. France. {\bf 52}, 49–85, (1924).

\bibitem{Le25} {\sc L\'evy, P.} Calcul des Probabilités. Gauthier–Villars, Paris. (1925).

\bibitem{MandrekarRudiger} {\sc Mandrekar, V. and  R\"udiger, B.} L\'evy noises and stochastic integrals on Banach spaces.
 {\it Stochastic partial differential equations and applications VII},
 193--213, Lect. Notes Pure Appl. Math., 245, Chapman \& Hall/CRC, Boca Raton, FL,  2006.


\bibitem{Marc2}
     {\sc Marcus. R}.
     Parabolic {I}t\^o equations,
    {\it Trans. Amer. Math. Soc.},
   {\bf 198},
      (1974),
     pp. 177--190.

\bibitem{Ma}
    {\sc Marcus. R.}
    Parabolic {I}t\^o equations with monotone nonlinearities,
   {\it J. Funct. Anal.},
     {\bf 29},
     (1978),
    no.\,3,
     pp. 275--286.


 \bibitem{MaRo} {\sc Ma, Z. M.,  R\"ockner, M.} Introduction to the theory of (nonsymmetric) Dirichlet forms.
Universitext. Springer-Verlag, Berlin,  1992. vi+209 pp. ISBN:
3-540-55848-9

\bibitem{MaTa}
    {\sc Malliavin. P and Taniguchi. S.}
     Analytic functions, {C}auchy formula, and stationary phase on
              a real abstract {W}iener space,
   {\it J. Funct. Anal.},
    {\bf 143},
      (1997),
   no.\,2,
     pp. 470--528.

\bibitem{Marin} {\sc Marinelli, C.}{\it Local well-posedness of Musiela's SPDE with L\'evy noise.}
 Math. Finance {\bf 20}  (2010),  no. 3, 341--363.

\bibitem{MarQuer}
{\sc Marinelli, C. and Quer-Sardanyons, L.} {\it Existence of weak
solutions for a class of semilinear stochastic wave equations} Siam
J.Math. Anal. {\bf 44}, pp. 906--925 (2012)

\bibitem{MARR} {\sc Marinelli, C. and R\"ockner, M.}
Uniqueness of mild solutions for dissipative stochastic
reaction-diffusion equations with multiplicative Poisson noise. {\it
Electron. J. Prob.}{\bf 15}, 1528-1555 (2010).

\bibitem{MBPR} {\sc Meyer-Brandis. T and Proske. F} {\it Explicit
representation of strong solutions of SDEs driven by infinite
dimensional L\'evy processes.} J. Theor. Prob. 23, 301-314 (2010).

\bibitem{S.Mitter}{\sc Mitter, Sanjoy K. } Stochastic quantization.
 Modeling and control of systems in engineering, quantum mechanics,
 economics and biosciences (Sophia-Antipolis, 1988),
 151--159, {\it  Lecture Notes in Control and Inform. Sci.},vol{\bf 121}, Springer, Berlin,  1989.

\bibitem{Mung} {\sc Mugnolo. D.} {\it Gaussian estimates for a heat equation on a network}, Netw. Heter. Media 2, 55-79,
2007.
\bibitem{Mung1} {\sc Mugnolo. D and Romanelli. S.} {\it Dynamic and
generalized Wentzell node conditions for network equations.} Math.
Meth. Appl. Sciences 30, 681-706, 2007.
\bibitem{Mumf} {\sc Mumford. D.} {\it The dawning of the age of
stochasticity.} Atti.Acc. Naz. Lincei (9), 107-125 (2000).
\bibitem{Pa}
    {\sc Parisi. G.}
     Statistical field theory,
     {\it Frontiers in Physics},
    {\bf 66},
   With a foreword by David Pines,
Benjamin/Cummings Publishing Co. Inc. Advanced Book Program,
    Reading, MA,
     (1988),
    pp. xvi+352.

\bibitem{Part} {\sc Parthasarathy, K. R.},  Probability measures on metric spaces.
{\it Probability and Mathematical Statistics}, {\bf No. 3}, Academic
Press, Inc., New York-London  (1967).

\bibitem{PeZa} {\sc Peszat. S and Zabczyk. J.} {\it Stochastic partial
differential equations with L\'evy noise.} Encyclopedia of
Mathematics and its applications 113, Cambridge University Press,
2007.

 \bibitem{ReSi} {\sc Reading, John F. ;  Sigel, James L.}{\it  Exact solution of the one-dimensional Schr\"odinger equation with
 $\delta$-function potentials of arbitrary position and strength.}
 Phys. Rev. B (3){\bf   5}  (1972),  no. 2, 556--565.



\bibitem{PrZa} {\sc Priola, Enrico and Zabczyk, Jerzy. } On linear evolution equations for a class of cylindrical L\'evy
 noises.
{\it  Stochastic partial differential equations and applications},
 223--242, Quad. Mat., 25, Dept. Math., Seconda Univ. Napoli, Caserta,  (2010)

\bibitem{PRRO} {\sc Pr\'evot. C and R\"ockner Micheal.} {\it A
Concise Course on Stochastic Partial Differential Equations}.
Springer Berlin Heidelberg. (2008).

 \bibitem{RS} {\sc Reed, M. and Simon, B.}
{\it Methods of Modern Mathematical Physics, II. Fourier Analysis,
Self- Adjointness}, Academic Press, San Diego, New York, 1975

\bibitem{RT}
    {\sc Rovira. C and Tindel. S.}
    Sharp {L}aplace asymptotics for a parabolic {SPDE},
  {\it Stochastics Stochastics Rep.},
   {\bf 69},
     (2000),
  no.\,1-2,
     pp. 11--30.

\bibitem{RozMi} {\sc Mikulevicius, R. and  Rozovskii, B.} {\it  Linear parabolic stochastic PDEs and Wiener chaos.}
 SIAM J. Math. Anal. {\bf 29}  (1998),  no. 2, 452--480.

\bibitem{RZI} {\sc R\"udiger. B and Ziglio. G.}
It\^o formula for stochastic integrals w.r.t compensated Poisson
random measures on separable Banach spaces. {\it Stochastics} {\bf
78}, 377-410 (2006).

\bibitem{Sat91} {\sc Sato. K.} {\it  L\'evy processes and infinite divisibility.} Cambridge University Press, 1999.

 \bibitem{Sato} {\sc Sato, K.}, Stochastic integration for L\'evy processes and infinitely divisible
 distributions.
(Japanese) {\it Sugaku}, {\bf 63},   (2011) no. 2, 161--181.

 \bibitem{Schi}
    {\sc Schilder. M.}
     Some asymptotic formulas for {W}iener integrals,
   {\it Trans. Amer. Math. Soc.},
    {\bf 125},
     (1966),
     pp. 63--85.

\bibitem{SchiSchnurr} {\sc Schilling, R. L. , Schnurr, A.},
The symbol associated with the solution of a stochastic differential
equation. {\it Electron. J. Probab.} {\bf 15}, (2010)

\bibitem{Schnurr} {\sc Schnurr, A.}, The symbol of a Markov Semimartingale, pp. 1-118, Diss. T.U. Dresden, (2008).

\bibitem{Si}
     {\sc Simon. B.}
    Functional integration and quantum physics,
    Second edition,
{\it AMS Chelsea Publishing}, Providence, RI,
      (2005),
     pp. xiv+306.

\bibitem{SY84} {\sc Sato. K and Yamazato. M.}, Operator
self-decomposable. {\it Stochastic processes and their
applications}, {\bf 17}. 73-100, (1984).






 \bibitem{TuEsp}
    {\sc Tuckwell. H. C.}
    Analytical and simulation results for the stochastic spatial
              {F}itz{H}ugh-{N}agumo model neuron,
   {\it Neural Comput.},
    {\bf 20},
      (2008),
   no.\,12,
     pp. 3003--3033.

\bibitem{Tu1}
    {\sc Tuckwell. H. C.}
     Introduction to theoretical neurobiology. {V}ol. 1, Linear cable theory and dendritic structure,
    {\it Cambridge Studies in Mathematical Biology},
    {\bf 8},
 Cambridge University Press,
   Cambridge,
      (1988),
     pp. \ xii+291.

\bibitem{Tu2}
    {\sc Tuckwell. H. C.}
    Introduction to theoretical neurobiology. {V}ol. 2, Nonlinear and stochastic theories,
    {\it Cambridge Studies in Mathematical Biology},
    {\bf 8},
Cambridge University Press,
   Cambridge,
     (1988),
     pp. xii+265.

 \bibitem{Tu92}
     {\sc Tuckwell. H. C.}
     Random perturbations of the reduced {F}itz{H}ugh-{N}agumo
              equation,
  {\it Phys. Scripta},
   {\bf 46},
      (1992),
    no.\,6,
    pp. 481--484.
\bibitem{Tu}
     {\sc Tuckwell. H. C and Jost. J.} {\it Moment analysis of the Hodgkin-Huxley system with additive noise.} Physica
     A, 388: 4115-4125, 2009.
\bibitem{Tu0}
     {\sc Tuckwell. H. C, Jost. J and Gutkin. B. S.} {\it Inhibition
     and modulation of rhythmic neuronal spiking by noise.} Physical
     Review E, 80(3):031907, 2009.
\bibitem{Wa}
    {\sc Watanabe. S.}
     Analysis of {W}iener functionals ({M}alliavin calculus) and
              its applications to heat kernels,
   {\it Ann. Probab.},
    {\bf 15},
      (1987),
    no.\,1,
     pp. 1--39.
\bibitem{Wal} {\sc Walsh. J. B.} {\it An introduction to stochastic
partial differentail equations.} In Ecole d'\'et\'e de
probabilit\'es de Saint-Flour, XIV-1984, volume 1180 of Lecture
Notes in Mathematics, pages 265-439. Springer, Berlin, 1986.

\bibitem{Yama} {\sc Yamazato, M.}  Absolute continuity of operator-self-decomposable distributions on $\R^d$,
{\it J. Multivariate Anal.} {\bf 13}, no. 4, (1983).


\bibitem{Za88} {\sc Zabczyk, J.}
{\it Symmetric solution of semilinear stochastic equations}
Proceedings of a Conference on Stochastic Partial Differential
Equations, Trento, Italy, 1987, Lecture Notes in Mathematics 1390
(1989), 237-256.






\end{thebibliography}
\end{document}